\documentclass[a4paper,twoside,11pt]{amsart}
\usepackage[english]{babel}
\usepackage{amsmath}
\usepackage{amssymb}
\usepackage{amsfonts}
\usepackage{stmaryrd} 
\usepackage{mathrsfs}
\usepackage{fixltx2e}[2005/12/01]
\usepackage{hyperref}
\usepackage{xcolor}

\usepackage{enumerate}

\let\itemref\ref

\usepackage[matrix,arrow,cmtip]{xy} 
\SelectTips{cm}{}
\newdir{(}{{}*!/-5pt/@^{(}}
\newdir{(x}{{}*!/-5pt/@_{(}}
\newdir{+}{{}*!/-9pt/{}}
\newdir{>+}{@{>}*!/-9pt/{}}
\entrymodifiers={+!!<0pt,\fontdimen22\textfont2>}

\usepackage[nosubsections]{mytheorems2}

\newtheoremstyle{dtheoremnopar}{3 mm}{1 mm}{\itshape}{}{\bfseries}{.}{ }
{\thmname{#1}\thmnumber{ #2}\thmnote{ \mdseries(#3)\bfseries}}

\theoremstyle{dtheoremnopar}
\newcounter{theoremx}
\newtheorem{theoremalpha}[theoremx]{Theorem}

\newcommand{\tref}[1]{\ref{#1}} 

\theoremstyle{ddef}
\newtheorem*{definition*}{Definition}

\newcommand{\ZZ}{\mathbb{Z}}
\newcommand\xdashrightarrow[1]{\overset{#1}{\dashrightarrow}}

\DeclareMathOperator{\id}{id}
\newcommand{\inj}{\hookrightarrow}
\newcommand{\surj}{\twoheadrightarrow}
\DeclareMathOperator{\image}{im}
\DeclareMathOperator{\kar}{char}

\DeclareMathOperator{\rk}{rk}
\DeclareMathOperator{\Ext}{Ext}
\DeclareMathOperator{\Tor}{Tor}
\DeclareMathOperator{\Hom}{Hom}
\newcommand{\sheafHom}{\operatorname{\mathcal{H}om}}
\newcommand{\sheafExt}{\operatorname{\mathcal{E}xt}}
\DeclareMathOperator{\Gal}{Gal} 
\DeclareMathOperator{\Sym}{Sym}

\newcommand\sE{\mathcal{E}}
\newcommand\sF{\mathcal{F}}
\newcommand\sG{\mathcal{G}}
\newcommand\sL{\mathcal{L}}
\newcommand\sN{\mathcal{N}}
\newcommand\sO{\mathcal{O}}
\newcommand\im{\mathfrak{m}} 

\newcommand\A[1]{\mathbb{A}^{#1}}    
\DeclareMathOperator{\Spec}{Spec}
\newcommand\red{\mathrm{red}}   
\newcommand{\etale}{\'{e}tale}

\DeclareMathOperator{\stab}{stab}
\DeclareMathOperator{\Mor}{Mor}

\DeclareMathOperator{\Exal}{Exal}

\newcommand{\gen}{{\mathrm{gen}}} 
\newcommand{\rig}{{\mathrm{rig}}}

\newcommand{\cs}{{\mathrm{cs}}} 
\newcommand{\can}{{\mathrm{can}}} 
\newcommand{\stG}{\mathcal{G}}

\newcommand{\VV}{\mathbb{V}} 

\newcommand{\QCoh}{\mathsf{QCoh}} 
\newcommand{\VB}{\mathsf{VB}} 
\newcommand{\B}{\mathrm{B}} 
\renewcommand{\H}{{\operatorname{H}}} 

\newcommand{\Gm}{\mathbb{G}_m}
\newcommand{\Gmu}{\pmb{\mu}} 
\newcommand{\Galpha}{\pmb{\alpha}}

\DeclareMathOperator{\K}{K} 
\newcommand{\triv}{{\mathrm{t}}} 
\newcommand{\nt}{{\mathrm{nt}}} 

\newcommand{\D}{\mathsf{D}} 
\newcommand{\Dqc}{\mathsf{D}_\mathrm{qc}} 
\newcommand{\Rd}{\mathsf{R}}
\newcommand{\Ld}{\mathsf{L}}

\newcommand{\sptag}[1]{\href{http://stacks.math.columbia.edu/tag/#1}{Tag~#1}}

\begin{document}

\title[Functorial destackification]
{Functorial destackification and weak factorization of orbifolds}
\author{Daniel Bergh}
\address{Department of Mathematical Sciences\\
         Copenhagen University\\
         Universitetspark~5\\
         2100~Copenhagen~\O{}\\
         Denmark}
\email{dbergh@gmail.com}
\author{David Rydh}
\address{KTH Royal Institute of Technology\\
         Department of Mathematics\\
         SE\nobreakdash-100~44~Stockholm\\
         Sweden}
\email{dary@math.kth.se}
\date{2019-05-02}
\hyphenation{Buon-erba}

\begin{abstract}
Let $X$ be a smooth and tame stack with finite inertia.  We prove that there is a
functorial sequence of blow-ups with smooth centers after which the stabilizers
of $X$ become abelian. Using this result, we can extend the destackification
results of the first author to any smooth tame stack. We give applications to
resolution of tame quotient singularities, prime-to-$\ell$ alterations of
singularities and weak factorization of Deligne--Mumford stacks. We also extend
the abelianization result to infinite stabilizers in characteristic zero,
generalizing earlier work of Reichstein--Youssin.
\end{abstract}

\maketitle


\begin{section}{Introduction}
Before carefully stating our main results and discussing some applications,
let us first briefly recall the notion of destackification from
\cite{bergh_functorial-destackification} and discuss it in a less formal setting.
For simplicity, we fix a field~$k$ of arbitrary characteristic and regard only separated,
tame Deligne--Mumford stacks of finite type over $k$ in this first part of the introduction.
A typical example of such a stack would be the quotient stack $[X/G]$ where
$G$ is a finite group with order invertible in $k$, and $X$ is a $G$-variety over~$k$.

A \emph{modification} of an algebraic space $X$ is a proper birational morphism
$X' \to X$ of algebraic spaces,
the typical example being a blow-up in a nowhere dense closed subspace.
Similarly, a \emph{stacky modification} of a stack $X$ is a proper
birational morphism $X' \to X$ of stacks.
Since blow-ups are local constructions, they make perfect sense also for algebraic stacks,
but there are also stacky modifications which are not representable.
A typical example of this is a \emph{root stack} in a Cartier divisor $D$ on~$X$.
By a \emph{stacky blow-up}, we will mean either an ordinary blow-up or a root stack in a divisor.

Under our assumptions, a stack $X$ admits a canonical morphism $X \to X_\cs$
to its coarse space.
This morphism is proper and initial in the category of algebraic spaces under~$X$.
Recall that even if $X$ is smooth, the same need not hold true for $X_\cs$ ---
in general $X_\cs$ will have finite tame quotient singularities.
A \emph{destackification} is a stacky modification $X' \to X$
such that $X'_\cs$ is smooth.
In particular, the induced morphism $X'_\cs \to X_\cs$ is a
resolution of singularities.

We are primarily interested in destackifications with additional nice properties.
Specifically, we ask our destackification $X' \to X$ to be the composition
of a sequence
\vspace{-2mm}
\begin{equation}
\label{eq-intro-destack}
X' = X_n \to \cdots \to X_0 = X
\end{equation}
of stacky blow-ups in smooth centers.
We also desire the canonical morphism $X' \to X'_\cs$ to be as nice as possible.
Recall that this morphism factors canonically through the rigidification $X' \to X'_\rig$.
The morphism $X' \to X'_\rig$ is a gerbe such that $X'_\rig$ contains an open dense
substack which is an algebraic space.
The destackifications we consider here have the additional property that $X'_\rig \to X'_\cs$ is an
\emph{iterated root stack} centered in a simple normal crossings divisor.

Destackification is analogous to resolution of singularities by a sequence of
smooth blow-ups. Instead of removing stackiness via
the singular coarse space $X\to X_\cs$, we have a sequence of smooth stacky
blow-ups and gerbes as in the previous paragraph
$X'_\cs \leftarrow X'_\rig \leftarrow X' \rightarrow \dots \rightarrow X$.
This helps us reduce questions about smooth stacks to questions about smooth
algebraic spaces if we understand how the question is affected by stacky blow-ups in smooth centers
and rigidifications.

In \cite{bergh_functorial-destackification} a destackification as in \eqref{eq-intro-destack}
is shown to exist provided that $X$ has abelian stabilizers.
This is not a severe restriction provided that the base field $k$ has characteristic zero.
Indeed, assume for simplicity that $X$ is generically an algebraic space.
A result by Reichstein--Youssin \cite[Theorem~4.1]{reichstein-youssin_ess-dim-resolution-thm} implies that if the stacky locus is contained in a simple normal
crossings divisor, then the stabilizers are automatically abelian.
Since we are in characteristic zero,
we can arrange for this to happen by using embedded resolution of singularities.
In positive characteristic, however, we do not have resolution of singularities and the abelian result is too restrictive for applications such as \cite{bls2016}.

Our first main result (Theorem~\ref{TM:divisorialification}) in this article is that we can abelianize the stabilizers
without using embedded resolution of singularities.
This gives us a destackification theorem (Theorem~\ref{TM:destackification}) which is more powerful than
the corresponding result \cite[Theorem~1.2]{bergh_functorial-destackification}.
In particular, it holds in positive characteristic and in the relative setting.
Our second main result (Theorem~\ref{TM:unipotent-stabs}) is a generalization of the result by
Reichstein--Youssin mentioned in the previous paragraph.
This gives the structure of the stabilizers of a smooth algebraic stack with
stabilizers of arbitrary dimension,
provided that the stacky locus is contained in a simple normal crossings divisor.

As reflected in the title of this article, we also intended to include a theorem on weak factorization
of orbifolds, i.e., Deligne--Mumford stacks which generically have trivial stabilizers, in characteristic zero.
This result is briefly mentioned in \cite{bergh_functorial-destackification} as Corollary~1.5.
Since then, a stronger result by Harper \cite{harper_factorization} on weak factorization of Deligne--Mumford stacks based
on destackification has appeared,
which makes our result less relevant.
Instead of discussing our previous result further,
we simply restate Harper's theorem as Theorem~\ref{TM:HARPER}.
We also give a minor simplification of its proof.
We decided to keep the title since this article has existed as a draft for
quite some time which is already referenced by other work~\cite{bls2016, kresch_restricted-destackification}.

\subsection*{Main results}
We will now make more detailed statements of our main results.
From now on, we will use the definitions of algebraic stacks and algebraic spaces given in
the stacks project \cite[\sptag{026O}, \sptag{025Y}]{stacks-project}.
In particular, we will not assume that our stacks are Deligne--Mumford stacks and we
will always state any separatedness conditions explicitly.  

We will usually work with stacks relative an arbitrary base and it will sometimes be convenient to let this
base be an algebraic stack.
To avoid tedious repetitions, we encapsulate some of our standard assumptions in the notion
of a \emph{standard pair} $(X, D)$ over the base~$S$.
This is an algebraic stack $X$ which is smooth and of finite presentation over~$S$,
together with an ordered simple normal crossings divisor on~$X$ (see Definition~\ref{def-snc}).

Our algebraic stacks will usually have \emph{finite inertia}.
Recall that this assumption implies that the stack in question has a coarse space.
Furthermore, the stabilizers will usually be tame in the sense that they are finite linearly reductive group schemes. 
An algebraic stack with finite inertia and tame stabilizers is what is called a \emph{tame Artin stack}
by Abramovich--Olsson--Vistoli in~\cite{abramovich_olsson_vistoli_tame_stacks}.
As we occasionally work with more general stacks in this article
(Theorem~\ref{TM:unipotent-stabs}), we do not assume that our standard pairs
are tame stacks as in
\cite[Definition 2.1]{bergh_functorial-destackification}.

Our first main theorem is a generalization of \cite[Theorem~1.1]{bergh_functorial-destackification}
to the situation with arbitrary tame stabilizers.
A more general formulation is given as Theorem~\ref{T:divisorialification}.
See Section~\ref{S:stacky-blowups} for details on standard pairs,
rigidifications, blow-up sequences and functoriality.

\begin{theoremalpha}[Divisorialification]\label{TM:divisorialification}
Let $S$ be an algebraic space and let $(X, D)$ be a standard pair over~$S$.
Assume that $X$ has finite inertia, tame stabilizers and relative dimension $\leq d$ over~$S$.
Then there exists a smooth ordinary blow-up sequence
\[
\Pi\colon (Y, E) = (X_n, D_n) \to \cdots \to (X_0, D_0) = (X, D)
\]
with $n \leq d$, such that $Y$ admits a rigidification
$r\colon Y \to Y_\rig$
and the geometric stabilizers of $Y_\rig$ are diagonalizable and
trivial outside $r(E)$.

The construction is functorial with respect to arbitrary base change $S' \to S$ and
with respect to smooth morphisms~$X' \to X$ which are either stabilizer preserving or tame gerbes.
\end{theoremalpha}

As a consequence of Theorem~\ref{TM:divisorialification},
we get the following stronger version of \cite[Theorem~1.2]{bergh_functorial-destackification}.
A more general formulation is given as Theorem~\ref{T:destackification}.
See Section~\ref{S:stacky-blowups} for details on standard pairs,
rigidifications, root stacks, stacky blow-up sequences and functoriality.

\begin{theoremalpha}[Destackification]\label{TM:destackification}
Let $S$ be a quasi-compact algebraic space and let $(X, D)$ be a standard pair over~$S$.
Assume that $X$ has finite inertia and tame stabilizers.
Then there exists a smooth stacky blow-up sequence
\begin{equation}
\label{eqn-destack}
\Pi\colon (Y, E) = (X_n, D_n) \to \cdots \to (X_0, D_0) = (X, D)
\end{equation}
such that the coarse space $(Y_\cs, E_\cs)$ is a standard pair
over~$S$.
Moreover, the stack $Y$ admits a rigidification $Y \to Y_\rig$
such that the canonical morphism $Y_\rig \to Y_\cs$
is an iterated root stack in $E_\cs$.

The construction is functorial with respect to arbitrary base change~$S' \to S$
and with respect to smooth morphisms~$X' \to X$ which are either stabilizer preserving or tame gerbes.
\end{theoremalpha}

\subsection*{Amplifications to destackification}
There are three important amplifications of the destackification theorem.

\begin{remark}[(Restricted root operations)]
Kresch~\cite{kresch_restricted-destackification} has recently proven the
following strengthening.
Let $S$ and $(X, D)$ be as in Theorem~\ref{TM:destackification}.
Furthermore, let $p$ be a prime and assume that $p$ does not divide the order of the
stabilizer at any point of $X$.
Then a modification of the construction in Theorem~\ref{TM:destackification} yields a stacky blow-up sequence in
which all root stacks have order prime to~$p$.
In particular, if $S$ is a field of characteristic~$p$ and we start with a tame Deligne--Mumford stack $X$, then all the stacks $X_i$ are tame Deligne--Mumford stacks.
The modified construction enjoys the same functorial properties as the construction in~Theorem~\ref{TM:destackification}.
\end{remark}

\begin{remark}[(Relative destackification)]
In Theorem~\ref{T:destackification} we formulate a relative version of the
destackification theorem. Instead of removing all the stackiness of $X$, we
only remove the stackiness relative to an arbitrary target stack $T$. The end
result is thus a stack $Y_{\cs/T}$ which is representable over $T$. This
amplification was inspired by Harper's proof of weak factorization for
Deligne--Mumford stacks, see \cite[Corollary~3.19.1]{harper_factorization}.
\end{remark}

\begin{remark}[(Projectivity)]\label{R:projectivity}
A stacky blow-up is \emph{projective}~\cite{kresch_geometry-of-DM-stacks} in
the sense that it is a relative global quotient stack with a projective
relative coarse space. To be precise, if $\Pi\colon (Y, E) = (X_n, D_n) \to
\cdots \to (X_0, D_0) = (X, D)$ is a stacky blow-up sequence and
$E_1,E_2,\dots,E_n\subseteq Y$ are the exceptional divisors of $\Pi$, then the
induced morphism $Y\to \B \Gm^n \times X$ is representable and there are
positive integers $N_1\gg N_2 \gg \dots \gg N_n$ such that the line bundle
$\sO(N_1E_1+N_2E_2+\dots+N_nE_n)$ descends to a relatively ample line bundle on
the relative coarse space $Y_{\cs/X}$.  It follows that $Y_\cs\to X_\cs$ is
projective (if $S$ is quasi-separated). In fact one can show that each
morphism $X_{i+1,\cs}\to X_{i,\cs}$ is a blow-up.
\end{remark}

\subsection*{Stabilizers of positive dimension}

Although not needed to prove the theorems above,
similar methods give the following theorem of independent interest, generalizing~\cite[Theorem~4.1]{reichstein-youssin_ess-dim-resolution-thm}.

\begin{theoremalpha}\label{TM:unipotent-stabs}
Let $(X, D)$ be a standard pair over~$\Spec k$ for an arbitrary field~$k$.
Assume that $X$ has separated diagonal
and that $X \smallsetminus D$ is an algebraic space.
Let $x\colon \Spec \overline{K} \to X$ be a geometric point in $|D|$.
Then the stabilizer of $x$ is a semi-direct product $U \rtimes \Delta$
where $\Delta$ is diagonalizable and $U$ is unipotent.
Furthermore, if $n$ components of $D$ passes through $x$, then the induced
representation $\Delta \to \Gm^n$ is faithful.
\end{theoremalpha}

\subsection*{Weak factorization}
Relative destackification reduces the problem of weak factorization of birational
morphisms of smooth Deligne--Mumford stacks to weak factorization of
representable birational morphisms, which is proven by
Abramovich--Karu--Matsuki--W\l{}odarczyk~\cite{abramovich-etal_weak-factorization}.
See \cite[\S 5]{harper_factorization} for details. Note that the relative
formulation of the destackification theorem given in
Theorem~\ref{T:destackification} simplifies Step~III of Harper's proof.

\begin{theoremalpha}[{Harper~\cite[Theorem~1.1]{harper_factorization}}]
\label{TM:HARPER}
Let $\varphi\colon Y \dashrightarrow X$
be a proper birational map of smooth, separated Deligne--Mumford stacks of finite type
over a field $k$ of characteristic zero.
Let $U \subset X$ be an open substack over which $\varphi$ is an isomorphism.
Then $\varphi$ factors as a sequence
\begin{equation}
\label{eq-intro-wf}
Y = X_n \xdashrightarrow{\varphi_n} \cdots \xdashrightarrow{\varphi_1} X_0 =: X
\end{equation}
of proper birational maps $\varphi_i$ satisfying the following properties:
\begin{enumerate}
\item
\label{it-wf-stacky-blow-up}
either $\varphi_i$ or $\varphi_i^{-1}$ is a stacky blow-up centered in a smooth locus;
\item
\label{it-wf-invariant-locus}
each $\varphi_i$ is an isomorphism over $U$;
\item
\label{it-wf-normal-crossings}
if the exceptional loci $D_i := X_i\smallsetminus U$ are simple normal crossing divisors for $i = 0, n$,
then the same holds true for all $i$,
and all stacky blow-ups have centers which have normal crossings with the respective~$D_i$.
\item
\label{it-wf-defined}
there is an integer $n_0$ such that $X_i \dashrightarrow X$ is everywhere defined whenever $i \leq n_0$ and
$X_i \dashrightarrow Y$ is everywhere defined whenever $i \geq n_0$;
\end{enumerate}
\end{theoremalpha}

For further remarks on the various versions of weak factorization of
Deligne--Mumford stacks, see~\cite[Remark~1.3 and Remark~1.4]{bergh_factorization}.

\subsection*{Applications to resolutions of singularities}
Finally we list some important applications to resolution of singularities which easily follow from the
destackification theorem.

The first result is on resolution of tame quotient singularities.
Let $k$ be a perfect field, and let $X$ be an integral scheme over~$k$.
We say that $X$ has \emph{finite tame quotient singularities} if \'etale-locally on~$X$,
there exists a finite linearly reductive group scheme~$G$ over~$k$
and $G$-scheme $U$, smooth over~$k$, such that~$X=U/G$.

\begin{theoremalpha}[Resolution of tame quotient singularities]\label{TM:RESOLUTION-OF-QUOT-SINGS}
Let $k$ be a perfect field and let $X$ be a variety over $k$ with finite tame
quotient singularities. Then there exists a sequence of blow-ups
\[
X'=X_n\to X_{n-1}\to \dots \to X_1\to X_0=X
\]
such that $X'/k$ is smooth. Moreover, this sequence is functorial with respect
to field extensions and smooth morphisms $X'\to X$.
\end{theoremalpha}
\begin{proof}
By a result by Satriano~\cite[Theorem~1.10]{satriano_CST-finite-lin-red},
the variety $X$ admits a (unique) canonical stack $X_\can \to X$.
That is, there is a smooth tame stack $X_\can$ with coarse space $X$ such that
the canonical morphism $X_\can \to X$ is an isomorphism over an open
subvariety $U \subset X$ such that $X \smallsetminus U$ has codimension $\geq 2$.
The theorem now follows by applying destackification (Theorem~\tref{TM:destackification})
to $X_\can$.
\end{proof}

A result similar to Theorem~\ref{TM:RESOLUTION-OF-QUOT-SINGS} has been obtained independently by
Buonerba without the use of destackification~\cite{buonerba_resolution-tame-quotient}.
Buonerba only considers varieties with quotient singularities that are tame and \'etale.

Combining destackification with de Jong's theory of
alterations~\cite{deJong_smoothness-semistability-alterations,deJong_families-of-curves-and-alterations},
we easily recover two other results on resolution of singularities.
Recall that a morphism $\pi\colon X' \to X$ of integral schemes is called an \emph{alteration}
if it is proper and generically finite.
Furthermore, if $G$ is a finite group and $X'$ is a $G$-scheme over $X$,
we say that $(X', G)$ is a \emph{Galois alteration}~(cf.~\cite[Situation~5.3]{deJong_families-of-curves-and-alterations})
if the naturally defined group homomorphism from $G$ to the Galois group
$\Gal(K(X')/K(X))$ is an isomorphism, that is, if $K(X')^G/K(X)$ is purely
inseparable.
We formulate a version of a result by de Jong suited to our needs. 

\begin{theorem}[{de Jong \cite[Corollary~5.15]{deJong_families-of-curves-and-alterations}}]
\label{thm-galois-alt}
Let $X$ be an integral scheme of finite type over a field,
and let $Z \subsetneq X$ be a closed subset.
Then there exists a Galois alteration $(X', G)$ of $X$ such that $X'$ is regular
and the inverse image of $Z$ in $[X'/G]$ is the support of a simple normal crossings divisor.
\end{theorem}

The following theorem is of course just a special case of Hironaka's famous result ---
the interesting point is that the proof is by entirely different methods. 

\begin{theoremalpha}[Abramovich--de Jong~\cite{abramovich-deJong_smoothness-semistability-torgeom}]\label{TM:A-DJ}
Let $X$ be an integral scheme of finite type over a field~$k$ of characteristic zero,
and let $Z \subsetneq X$ be a closed subset.
Then there exists a modification $\pi\colon X' \to X$ such that $X'$ is smooth
and the inverse image of $Z$ is the support of a simple normal crossings divisor.
\end{theoremalpha}

\begin{proof}
Find a Galois alteration as in Theorem~\ref{thm-galois-alt}.
Since we assume that the characteristic is zero,
the stack quotient $[X'/G]$ is automatically smooth and tame,
and $X'/G$ is a modification of~$X$.
The result follows by destackification (Theorem~\tref{TM:destackification})
of the standard pair $([X'/G], D)$,
where $D$ is the simple normal crossings divisor whose support is the inverse image of~$Z$ in~$[X'/G]$.
\end{proof}


\begin{theoremalpha}[{Gabber 2005~\cite[Expos\'e~X, Theorem~2.1]{travaux-de-Gabber}}]\label{TM:GABBER}
Let $X$ be an integral scheme of finite type over a field $k$,
and let $Z \subsetneq X$ be a closed subset.
Let $\ell\neq \kar k$ be a prime number.
Then there exists an
alteration $\pi\colon X'\to X$ of degree prime to $\ell$ such that $X'$ is
regular and $\pi^{-1} (Z)$ is the support of a simple normal crossings divisor.
Moreover, $X'$ is smooth over a finite purely inseparable field extension
$k'/k$.
\end{theoremalpha}
\begin{proof}
Assume first that $k$ is perfect.
Find a Galois alteration as in Theorem~\ref{thm-galois-alt},
and let $H$ be an $\ell$-Sylow subgroup of $G$.
Then the stack $[X'/H]$ is smooth and tame,
and $X'/H \to X$ is an $\ell'$-alteration, i.e., the degree is prime to $\ell$.
The result follows by destackification (Theorem~\tref{TM:destackification}) of the standard pair $([X'/H], D)$,
where $D$ is the simple normal crossings divisor whose support is the inverse image of~$Z$ in~$[X'/H]$.
For general $k$, we first find an alteration of $\bigl(X\times_k
\Spec k^\mathrm{perf}\bigr)_\red$ with the desired properties.
This alteration is defined over some finite purely inseparable field extension $k'/k$.
\end{proof}

The original proofs of Theorem~\ref{TM:A-DJ} and Theorem~\ref{TM:GABBER}
by Abramovich--de Jong and Gabber, respectively,
are along the same lines as the proofs presented here.
But instead of using destackification,
the finite quotient singularities are resolved by first replacing $X'$
by a suitable blow-up such that $G$ (resp.\ $H$) acts toroidally on~$X'$.
Then $X'/G$ becomes toroidal and the singularities can be resolved using~\cite{KKMS}.

Recently, Temkin has improved Gabber's prime-to-$\ell$ alterations theorem and
obtained $p$-alterations~\cite{temkin_distillation}.
It would be interesting to find a proof of this using stacks,
cf.\ Remark~\ref{R:NpS}.

\subsection*{Outline}
In Section~\ref{S:preliminaries} we give some preliminaries on gerbes,
coarsenings and rigidifications.
In
Section~\ref{S:stacky-blowups}, we define standard pairs and recall the notion
of stacky blow-ups
from~\cite{bergh_functorial-destackification}.
In Section~\ref{S:charts} we give charts for smooth tame stacks:
\'etale-locally such a stack is simply a linear representation of a tame group
scheme.  In Section~\ref{S:codimension} we define our main technical invariant:
the \emph{codimension of stackiness}. We give a flexible definition via the
cotangent complex. In the case of a linear representation it is the dimension
of the non-trivial subspace of the representation. It is also indeed exactly
the codimension of the locus of maximal stabilizer.

In Section~\ref{S:divisorialification} we introduce our main invariant, the
\emph{divisorial index}, via codimension of stackiness. We then prove
divisorialification,
Theorem~\ref{TM:divisorialification}/\ref{T:divisorialification}.
In Section~\ref{S:destackification},
we combine Theorem~\ref{TM:divisorialification} with destackification for
stacks with diagonalizable stabilizers~\cite{bergh_functorial-destackification}
and prove destackification,
Theorem~\ref{TM:destackification}/\ref{T:destackification}.
In Section~\ref{S:infinite-stabilizers}
we prove Theorem~\ref{TM:unipotent-stabs} on the structure of stabilizers
when the stackiness is contained in a simple normal crossings divisor.

There are also two appendices on tame gerbes. In Appendix~\ref{A:gerbes}, we
show that every quasi-coherent sheaf on a tame gerbe splits canonically as a
direct sum of a sheaf with trivial action and a sheaf with purely non-trivial
action. We also give a similar result for complexes. In
Appendix~\ref{app-mult-type} we prove that the class of the cotangent complex
$L_{\B G/k}$ in the Grothendieck group $\K_0(\B G)$ is zero when $G/k$ is tame
and $k$ is algebraically closed.

\subsection*{Acknowledgments}
The first author was supported by the Danish National Research Foundation through the
\emph{Centre for Symmetry and Deformation} (DNRF92).
The second author was supported by the Swedish Research Council (2011-5599 and
2015-05554).

\end{section}
\setcounter{secnumdepth}{3}


\begin{section}{Preliminaries on stacks}\label{S:preliminaries}


\subsection{Inertia}
The \emph{relative inertia stack} $I_{X/S}$ of a morphism between algebraic
stacks $X\to S$ is defined as the
pull-back of the diagonal $\Delta\colon X \to X\times_S X$ along itself.
The stack $I_{X/S}$ has a natural structure of group object
over $X$ and this group is trivial if and only if the structure morphism
$X \to S$ is representable by algebraic spaces. A composition
$X\to Y\to S$ gives rise to an exact sequence of groups:
\begin{equation}\label{E:inertia-sequence}
1 \to I_{X/Y} \to I_{X/S} \to f^* I_{Y/S}
\end{equation}
We say that $X\to Y$ is \emph{stabilizer preserving} or \emph{inert} over $S$
if $I_{X/S}\to f^*I_{Y/S}$ is an isomorphism.

\subsection{Coarse spaces}
An algebraic stack $X$ with finite inertia $I_X$ admits a \emph{coarse (moduli)
  space}~$\pi\colon X \to X_\cs$, which is initial among morphisms to algebraic
spaces~\cite{keel_mori_quotients,rydh_finite_quotients}. 
Since the formation of coarse spaces respects flat base change,
this generalizes to the relative setting as follows. If $X \to S$ is a morphism
of algebraic stacks with finite relative inertia $I_{X/S}$, then it factors
canonically through its \emph{relative coarse space} $X \to X_{\cs/S} \to S$.
The morphism $X_{\cs/S}\to S$ is representable and the factorization is initial
among factorizations $X\to Y\to S$ with $Y\to S$ representable.

When $Y$ is a tame stack, a smooth morphism $X\to Y$ is stabilizer preserving over $S$ if and only if
$X_\cs\to Y_\cs$ is smooth and $X=X_\cs\times_{Y_\cs} Y$,
see~\cite[Proposition~6.7]{rydh_finite_quotients} for the \'etale case and \cite{rydh_luna-fundamental}
for the smooth case. Tameness is not needed in the \'etale case but necessary
in the smooth case.

\subsection{Gerbes}
An \emph{fppf gerbe} is a morphism $\pi\colon X\to S$ between stacks such
that for any $S$-scheme $T$ the following holds: (1) there exists a section
fppf-locally on $T$ and (2) two sections over $T$ are isomorphic fppf-locally
on $T$. Equivalently,
$\pi$ and $\Delta_\pi$ are epimorphisms of fppf stacks. If $X$ and $S$ are
algebraic stacks, then $\pi$ is smooth and $\Delta_\pi$ is faithfully flat and
locally of finite presentation.

Every stack $X$ is an fppf gerbe over its associated space $\pi_0(X)$: the fppf
sheafification of the associated presheaf of \emph{sets} $T\mapsto \pi_0\bigl(
X(T) \bigr)$. If $X$ is an algebraic stack, then $\pi_0(X)$ is an algebraic
space if and only if the inertia $I_X\to X$ is flat and locally of finite
presentation. An
\emph{(absolute) gerbe} is an algebraic stack $X$ such that $I_X\to X$
is flat and locally of finite presentation.

\subsection{Rigidifications}
\label{sec-rigid}
Let $X \to S$ be a morphism of algebraic stacks.
By a (partial) \emph{rigidification} of $X$ over~$S$,
we mean a factorization $X \to X' \to S$ such that the first morphism is a gerbe.
Such factorizations are in one to one correspondence with subgroups $K$
of the relative inertia $I_{X/S}$ which are flat and locally of finite presentation~\cite[Appendix~A]{abramovich_olsson_vistoli_tame_stacks}. Moreover, the
$2$-category of rigidifications of $X$ is equivalent to the partially ordered
set of flat subgroups.

We will only be interested in the case when both $K$ and $I_{X/S}$ are finite,
in which case we have a factorization $X \to X' \to X_{\cs/S}$. If there is a
final object among rigidifications, then we will call it the \emph{total
rigidification} and denote it $X_{\rig/S}$. The total rigidification of $X$ over
$S$ exists when $X$ is normal and irreducible and Deligne--Mumford over $S$: it
corresponds to
the union $I^{\gen}_{X/S}$ of the irreducible components of $I_{X/S}$ that dominate $X$. Note that $I^{\gen}_{X/S}\to X$ is \'etale because it is unramified and $X$ is
normal~\cite[Corollaire~18.10.3]{egaIV}.
If $X\to X'\to S$ is a rigidification of stacks with finite inertia over $S$
such that $X'\to S$ is representable
over a schematically dense open substack $U'\subseteq X'$, then $X'=X_{\rig/S}$.
Indeed, then $I_{X'/S}\times_{X'} X$ is the closure of
$I_{X/S}|_{U'\times_{X'} X}$ in $I_{X/S}$ and this is the largest flat subgroup.

\subsection{Coarsenings}
A morphism $f\colon X\to Y$ is called a \emph{coarsening} morphism if $f$ is a
universal homeomorphism with finite diagonal such that $f_*\sO_X=\sO_Y$. We
also say that $Y$ is a (partial) coarsening of $X$. Rigidifications $X\to X'$
with finite inertia and coarse space maps $X\to
X_{\cs/S}$ are examples of coarsenings~\cite[Theorem~6.12]{rydh_finite_quotients}.

The $2$-category of coarsenings of $X$ is equivalent to a partially ordered set
with initial object $X$ and final object $X_\cs$~\cite[\S
  2]{abramovich-temkin-wlodarczyk_destackification}. Not every coarsening
corresponds to a subgroup of inertia. There is an example of a coarsening
morphism $X\to Y$, with $X$ a wild Deligne--Mumford stack and $Y$ an algebraic
stack with finite inertia but not Deligne--Mumford, such that
$I_{X/Y}=0$~\cite[\S4.5]{romagny-rydh-zalamansky_complexity}.
But coarsening morphisms
$X\to Y$ such that $Y$ is Deligne--Mumford with finite inertia is in one-to-one
correspondence with open and closed subgroups of the inertia stack of
$I_X$~\cite[Appendix]{abramovich-temkin-wlodarczyk_destackification}.

The canonical factorization $X\to X_{\cs/S}\to S$ is not only initial among
factorizations where the second map is representable but also final among
factorizations where the first map is a coarsening morphism. That is,
coarsening morphisms and representable morphisms constitute an orthogonal
factorization system in the appropriate $2$-categorical sense. In particular,
$X\to S$ is a coarsening morphism if and only if $X_{\cs/S}=S$.
\end{section}


\begin{section}{Stacky blow-up sequences}\label{S:stacky-blowups}
In this section, we recall the notions of standard pairs and stacky blow-up sequences
from~\cite[Section~2]{bergh_functorial-destackification}.
 
\begin{definition}[Simple normal crossings divisor]
\label{def-snc}
Let $\pi\colon X \to S$ be a smooth morphism of algebraic stacks.
\begin{enumerate}
\item
A family $\{D^i \hookrightarrow X\}_{i \in I}$ of closed immersions
is called a \emph{labeled simple normal crossings divisor} on $X$ over~$S$
provided that $\bigcap_{j \in J} D^j$ is smooth of codimension $|J|$ over $S$
for every finite subset $J \subseteq I$.
\item
If, in addition, the indexing set $I$ is endowed with a well-ordering,
we say that $\{D^i \hookrightarrow X\}_{i \in I}$ is an
\emph{ordered simple normal crossings divisor} on $X$ over~$S$.
\item
The \emph{underlying divisor} of a labeled or ordered simple normal crossings divisor
$\{D^i \hookrightarrow X\}_{i \in I}$, with locally finite indexing set, is the closed substack $D = \bigcup_{i \in I} D^i$.
\item
A \emph{simple normal crossings divisor} $D$ on~$X$ relative~$S$ is a closed substack
of $X$ which is the underlying divisor of some labeled simple normal crossings divisor.
\item
A closed substack $Z \hookrightarrow X$ has \emph{normal crossings} with a given
labeled simple normal crossings divisor $\{D^i \hookrightarrow X\}_{i \in I}$
provided that its intersection with $\bigcap_{j \in J} D^j$ is smooth over~$S$
for each finite subset $J \subseteq I$.
\end{enumerate}
By abuse of notation, we usually denote a labeled or ordered simple normal crossings divisor
simply by~$D$.
\end{definition}

In this article, we will only consider ordered simple normal crossings divisors with finite indexing sets. 

\begin{remark}
Let $D = \{D^i \hookrightarrow X\}_{i \in I}$ be a labeled simple normal crossings divisor on $X$
relative~$S$ as in the definition above.
Then it is clear that each $D^i \hookrightarrow X$ is a smooth relative effective Cartier divisor on $X$
over~$S$, and $Z$ is smooth over $S$ if it has normal crossings with~$D$.
The condition that $Z$ has normal crossings with $D$ only depends on the underlying divisor
of~$D$, so the notion also makes sense for unlabeled simple normal crossings divisors.
\end{remark}

\begin{remark}
Note that we allow the divisors $D^i$ to be empty or to have multiple connected components.
This feature is crucial to obtain the right functoriality properties for our algorithms.
\end{remark}

\begin{remark}
\label{rem-snc-morphism}
Let $\pi\colon X \to S$ be a morphism of algebraic stacks,
and let $D = \{D^i \hookrightarrow X\}^n_{i = 1}$ be a sequence of effective Cartier divisors on~$X$.
Then we get an induced sequence $\{\mathcal{O}_X \to \mathcal{O}_X(D^i)\}_{i = 1}^n$ of
line bundles together with global sections.
Such a datum corresponds to a morphism
\begin{equation}
\label{eq-snc-morphism}
\delta_{D/S} \colon X \to [\mathbb{A}^n_S/\Gm^n]
\end{equation}
over~$S$.
Note that $D$ has simple normal crossings relative $S$ if and only if the morphism $\delta_{D/S}$ is smooth.
\end{remark}

\begin{definition}[Standard pair]
\label{def-standard-pair}
Let $S$ be an algebraic stack.
A \emph{standard pair} $(X, D)$ is a pair where $X$ is an algebraic stack that is smooth and of finite
presentation over~$S$ and $D$ is an ordered simple normal crossings divisor on $X$ relative~$S$ indexed
by a finite set.
\end{definition}

In the rest of the section we describe our conventions on how standard pairs transform
under certain operations.
We fix an algebraic stack $S$ and a standard pair $(X, D)$ over~$S$,
with~$D = \{D^i \hookrightarrow X\}_{i \in I}$.

\subsection{Base change of standard pairs}
\label{S:sp-base-change}
There is an obvious notion of base change of $(X, D)$ along a morphism $S' \to S$
of algebraic stacks.
This is the standard pair $(X_{S'}, D_{S'})$ over~$S'$ where
$X_{S'} = X \times_S S'$, $D^i_{S'} = D^i \times_S S'$
and $D_{S'} = \{D_{S'}^i \hookrightarrow X_S\}_{i \in I}$.
Similarly, we may base change along a smooth morphism $X' \to X$
and obtain a standard pair $(X', D')$ over~$S$
with $D' = \{D_{X'}^i \hookrightarrow X'\}_{i \in I}$.

\subsection{Coarsenings of standard pairs}
\label{S:sp-coarse}
Given a morphism $X \to T$ over~$S$ such that $I_{X/T}$ is finite,
we may form the pair $(X_{\cs/T}, D_{\cs/T})$ over~$S$,
where we take $D_{\cs/T} = \{D_{\cs/T}^i \hookrightarrow X_{\cs/T}\}_{i \in I}$
with $D_{\cs/T}^i \hookrightarrow X_{\cs/T}$ the schematic image of $D^i \hookrightarrow X$ in~$X_{\cs/T}$.
The pair $(X_{\cs/T}, D_{\cs/T})$ is of course usually not a standard pair ---
the whole point with destackification being to try to reach a situation where it is.
If the fibers of $I_{X/T}$ are tame,
then at least the pair $(X_{\cs/T}, D_{\cs/T})$ has some good properties:
\begin{enumerate}
\item
$X_{\cs/T}$ is flat and of finite presentation over~$S$;
\item
$D^i_{\cs/T}$ is the relative coarse space of $D^i \to T$; 
\item
$D^i_{\cs/T} \hookrightarrow X_{\cs/T}$ is a flat family of $\mathbb{Q}$-Cartier divisors over $S$.
\end{enumerate}
These are easy consequences of the definitions, \cite[Theorem 3.2, Corollary 3.3]{abramovich_olsson_vistoli_tame_stacks} and \cite[Proposition~6.1]{olsson_G-torsors}.
Note that without the tameness hypothesis, none of these properties hold in general.

\subsection{Rigidifications of standard pairs}
\label{S:sp-rigidification}
Now let $X \to T$ be a morphism over~$S$ and assume that a
total rigidification $\pi\colon X \to X_{\rig/T}$, as in Section~\ref{sec-rigid}, exists.
In particular,  the morphism $\pi$ is a coarsening,
so we get an induced pair,
which we denote by $(X_{\rig/T}, D_{\rig/T})$,
as in the previous subsection.
This is always a standard pair over~$S$,
even if the relative inertia of $\pi$ is not tame.
This follows from the facts that gerbes are faithfully flat
(in fact, even smooth)
and that $D^i=\pi^{-1}(D_{\rig/T}^i)$.

\begin{warning}
Note our slightly misleading abuse of notation:
the canonical morphism $D^i\to D^i_{\rig/T}$ is typically only a partial
rigidification and not the total rigidification.
\end{warning}

\subsection{Blow-ups of standard pairs}
\label{S:sp-blow-up}
As blow-ups are local constructions, they make sense also for algebraic stacks.
Given our standard pair $(X, D)$,
we will only consider blow-ups in (smooth) centers $Z$
that have normal crossings with~$D$.
This produces a new standard pair $(\widetilde{X}, \widetilde{D})$.
Here $\widetilde{D} = \{\widetilde{D}^i \hookrightarrow \widetilde{X}\}_{i \in I^+}$,
where $I^+$ is $I$ with an extra element ${+}$ appended,
$\widetilde{D}^+$ is the exceptional divisor of the blow-up and $\widetilde{D}^i$ is the strict transform of $D^i$
for every $i \in I$.
We call $(\widetilde{X}, \widetilde{D})$ the \emph{transform} of the blow-up centered in $Z$.

\subsection{Root stacks}
\label{S:sp-root}
Root stacks were first studied systematically by Cadman~\cite[\S2]{cadman_root-stacks}
and play a vital role in the destackification process. 
Given an algebraic stack $Y$ and an effective Cartier divisor~$E$ on~$Y$,
the \emph{root stack} $\widetilde{Y} \to Y$ of order~$r$ is
obtained by formally adjoining an $r$th root $\widetilde{E}$ of $E$.
We say that the pair $(E, r)$ is the \emph{center} of the root stack.

We will only consider root stacks centered in one of the divisors $D^j$
of our standard pair $(X, D)$.
The construction produces a new standard pair $(\widetilde{X}, \widetilde{D})$.
Here we define $\widetilde{D}$ as $\{\widetilde{D}^i \hookrightarrow \widetilde{X}\}_{i \in I}$ with
$\widetilde{D}^j$ being the $r$th root of $D^j$ and $\widetilde{D}^i = D^i\times_X \widetilde{X}$
for $i \neq j$.
Note that the morphism $\widetilde{X} \to X$ is a coarsening and that
$(X, D)$ coincides with the standard pair $(\widetilde{X}_{\cs/X}, \widetilde{D}_{\cs/X})$
as defined in Section~\ref{S:sp-coarse}.
We call $(\widetilde{X}, \widetilde{D})$ the \emph{transform} of the root stack centered in $Z = (D^j, r)$.
By a \emph{stacky blow-up} of a standard pair we mean either a root stack or an ordinary blow-up
as described in Section~\ref{S:sp-blow-up}.

\begin{remark}
A natural way to combine blow-ups and root stacks is to first blow-up a closed
substack $Z\subset X$ and then take the $r$th root stack of the exceptional
divisor $E$. This is called a stacky blow-up with center $(Z,r)$
in~\cite{rydh_compactification-tame-stacks}. When $r=1$, this is an ordinary
blow-up and when $Z$ is a divisor, this is a root stack.  If $I$ is the ideal
sheaf defining $Z$, then the stacky blow-up with center $(Z,r)$ equals the
stacky Proj of the generalized Rees algebra $\bigoplus_{d\geq 0} I^{\lceil d/r
  \rceil}$.
Note however that this point of view is not quite compatible with the conventions for transforms of standard pairs
for root stacks used in this article.
\end{remark}
 
\subsection{Stacky blow-up sequences}
\label{S:sp-blow-up-sequence}

A \emph{stacky blow-up sequence} starting at $(X, D)$ is a sequence
$$
\Pi\colon (X_n, D_n) \xrightarrow{\pi_n} \cdots \xrightarrow{\pi_1} (X_0, D_0) = (X, D)
$$
of smooth stacky blow-ups in centers $(Z_i,r_i)$,
where $(X_{i + 1}, D_{i+1})$ denotes the transform of $(X_i, D_i)$.
We consider the centers of the stacky blow-ups to be part of the datum. 
Two such sequences are considered the same if they only differ by insertion or deletion
of stacky blow-ups with empty centers.
If all stacky blow-ups in the sequence are in fact ordinary blow-ups,
we say that $\Pi$ is an \emph{ordinary blow-up sequence}.

Just as for standard pairs there are obvious notions of base change of a stacky blow-up sequences
starting at $(X, D)$.
As before, we can consider two cases: base change along an arbitrary morphism $S' \to S$
and base change along a smooth morphism $X' \to X$.

Consider a construction which to each algebraic stack $S$ and each standard pair $(X, D)$ over $S$
associates a blow-up sequence $\Pi$ starting at $(X, D)$.
We say that the construction is \emph{functorial} with respect to a base change $S' \to S$
if the blow-up sequence obtained by applying the construction to~$(X_{S'}, D_{S'})$
equals the base change of $\Pi$ along $S' \to S$.
We say that the construction is \emph{functorial} with respect to a smooth morphism $X' \to X$
if the blow-up sequence obtained by applying the construction to~$(X', D')$
equals the base change of $\Pi$ along $X' \to X$.

\end{section}


\begin{section}{Charts for smooth tame stacks}\label{S:charts}
In this section we generalize the local charts
of~\cite[Proposition~5.3]{bergh_functorial-destackification} from stacks with
diagonalizable stabilizers to stacks with finite linearly reductive
stabilizers.

Let $G\to S$ be a group scheme.
A linear action of $G$ on $\A{n}_S$ is equivalent to giving a rank $n$ vector
bundle $\sE$ on $\B G_S$ together with a trivialization of $s^*\sE$ where $s$
denotes the
tautological section $S\to \B G_S$. This equivalence is
given by $[\A{n}_S/G]=\VV_{\B G}(\sE)$. If $G$ is tame, then $\sE=\sE_\triv\oplus \sE_\nt$ (Theorem~\ref{T:tame-splitting:QCoh}) and the
$G$-fixed locus of $\A{n}_S$ is $\VV_S(\sE_\triv)$ which is smooth of codimension
$\rk \sE_\nt$. The conormal bundle of the fixed locus is $s^*\sE_\nt$.
That a coordinate hyperplane $\{x_i=0\}$ is $G$-invariant means that
$\sE$ has a
corresponding rank $1$ direct summand.

\begin{theorem}\label{T:Luna-slice-finite}
Let $X$ be an algebraic stack with finite tame inertia and smooth over a
scheme $S$. Let $x\in |X|$ be a point and $s\in |S|$
its image. If $\kappa(x)/\kappa(s)$ is finite and separable, then there exists
a commutative diagram
\[
\xymatrix{%
([\A{n}_{S'}/G],0)\ar[d] & (W,w)\ar[l]_-q\ar[r]^-{p} & (X,x)\ar[d] \\
(S',s')\ar[rr] && (S,s)}
\]
where the horizontal morphisms are \'etale and stabilizer preserving and $G\to
S'$ is a finite locally free tame group scheme acting linearly
on $\A{n}_{S'}$.
If in
addition $D=\sum_i D^i$ is a divisor on $X$ with simple normal crossings
over $S$ (Definition~\ref{def-snc}),
such that $x\in D^i$ for all
$i$, then it can be arranged that $E^i=\{x_i=0\}\subset \A{n}_{S'}$ is
$G$-invariant and $q^{-1}(E^i)=p^{-1}(D^i)$.
\end{theorem}
\begin{proof}
Since the structure morphism $\stG_x\to \Spec \kappa(x)$ of the residual gerbe
is smooth, we can find a finite separable extension $k/\kappa(x)$ such that
$\stG_x(k)\neq \emptyset$. After replacing $(S,s)$ with an \'etale
neighborhood, we can thus assume that $k=\kappa(x)=\kappa(s)$ and $\stG_x=\B G_x$
where $G_x\to \Spec k$ is a finite tame group scheme. After replacing $(S,s)$
with an \'etale neighborhood, we may assume that $G_x$ is the restriction of a
finite locally free tame group scheme $G\to
S$~\cite[Proposition~2.18]{abramovich_olsson_vistoli_tame_stacks}. We can also assume
that $S$ is affine. \'Etale-locally on the coarse space $X_{\cs}$, we may then
write $X=[U/G]$ for an affine scheme $U$ over
$S$~\cite[Proposition~3.6]{abramovich_olsson_vistoli_tame_stacks} (follow Step 1, but
in the the first paragraph use the tame group scheme given as the pull-back of
$G\to S$ to $X_{\cs}$).

Next, we consider the conormal bundle $\sE_s:=\sN_{\stG_x/X_s}=\im/\im^2$ of
the closed immersion $i\colon \stG_x\inj X_s$. This is a locally free sheaf of
rank $n=\dim X_s$ on $\B G_x$. If $D=\sum_{i=1}^m D^i$ is a divisor with
simple normal crossings and
every component passes through $x$, then we let
$\sL_{s,i}=i^*\sN_{D^i/X}$ --- a locally free sheaf of rank
$1$. This gives rise to a splitting
\[
\sE_s=\sL_{s,1}\oplus \dots \oplus\sL_{s,m}\oplus \sE'_s
\]
where $\sE'_s=\sN_{\stG_x/\cap_i (D^i)_s}$.
By Lemma~\ref{L:deform-bundle-tame-gerbe}, we may lift these vector bundles to
$\B G$, after replacing $(S,s)$ with an \'etale neighborhood, obtaining a locally
free sheaf of rank $n$
\[
\sE=\sL_1\oplus \dots \oplus\sL_m\oplus \sE'
\]
on $\B G$. Equivalently, we have a linear action of $G$ on $\A{n}_S$ together
with $m$ invariant divisors $E^i=\{x_i=0\}$ and $\VV(\sE):=\Spec_{\B G} \Sym(\sE)
= [\A{n}_S / G]$.

Let $r\colon X=[U/G]\to \B G$ denote the structure morphism. Giving a map
$q\colon X\to \VV(\sE)$ is then equivalent to giving a ring homomorphism
$\Sym(\sE)\to r_*\sO_X$ in $\QCoh(\B G)$, that is, a homomorphism of modules
$\sE\to r_*\sO_X$ in $\QCoh(\B G)$. The surjection $\sE\to \sE_s=\im/\im^2$
lifts to a map $\sE\to r_*\im \to r_*\sO_X$ since $\sE$ is projective.
More precisely, for each $i$ we lift $\sL_i\to \sL_{s,i}=i^*\sN_{D^i/X}$ to
$\sL_i\to r_*\sN_{D^i/X}$ and we lift $\sE'\to \sE'_s\to \im/\im^2$ to $\sE'\to
r_*\im$.  By construction, $q$ is then \'etale at $x$ (this can be checked in
the fiber over $s$) and $q^{-1}(E)=p^{-1}(D)$ in a neighborhood of $x$.  We
conclude by replacing $X$ with an open neighborhood of $x$.
\end{proof}

\end{section}


\begin{section}{Codimension of stackiness}\label{S:codimension}
In this section, we introduce the notion of \emph{codimension of stackiness} at a point on a smooth, tame Artin stack,
and investigate its basic properties.
Roughly, the codimension of stackiness is defined as the dimension of the non-trivial part of the linear representation of the stabilizer on the tangent space at the point in question. 
Here some care must be taken, since the tangent space of a point on an Artin stack need not be a
vector space, but rather a Picard stack.
To avoid working with representations on such stacks,
we will instead take advantage of the fact that the necessary information is encoded into the cotangent complex.
We will also work in a relative sense,
since this will make our proofs in the next section easier.

Let $\mathcal{G}$ be an (absolute) gerbe with finite tame inertia,
and let $\mathcal{E}$ be a finite locally free sheaf on $\mathcal{G}$.
Recall that we have a canonical splitting $\mathcal{E} = \mathcal{E}_\triv  \oplus \mathcal{E}_\nt$
into a trivial and a non-trivial part (Theorem~\ref{T:tame-splitting:QCoh})
and a canonical splitting $\mathcal{P} = \mathcal{P}_\triv \oplus \mathcal{P}_\nt$
for any perfect complex on $\mathcal{G}$ (Theorem~\ref{T:tame-splitting:Dqc}).
Recall that the rank of a perfect complex is a locally constant function.
In particular, the following definition makes sense.

\begin{definition}
\label{def-relevant-rank}
Let $\mathcal{G}$ be a connected gerbe with finite tame inertia.
The \emph{relevant rank} of a perfect complex $\mathcal{P}$
on $\mathcal{G}$ is defined as the rank of the complex $\mathcal{P}_\nt$,
as defined above.
\end{definition}

\begin{lemma}
\label{lem-relevant-bs}
Let $f\colon \mathcal{G'} \to \mathcal{G}$ be a stabilizer preserving morphism of connected gerbes with finite tame inertia,
and let $\mathcal{P}$ be a perfect complex on~$\mathcal{G}$.
Then the relevant rank of $f^*\mathcal{P}$ with respect to $\mathcal{G}'$ equals the relevant rank of $\mathcal{P}$
with respect to~$\mathcal{G}$. 
\end{lemma}
\begin{proof}
Since $f$ is a stabilizer preserving morphism of gerbes,
the induced commutative square formed by the map between the coarse spaces is cartesian.
Let $\pi\colon \mathcal{G} \to \mathcal{G}_\cs$ denote the canonical map to the coarse space.
Then the relevant rank of a perfect complex $\mathcal{P}$ is equal to $\rk \mathcal{P} - \rk \pi_*\mathcal{P}$.
In particular, the result follows since the rank of a perfect complex is preserved under pull-back
and $\pi_*$ respects tor-independent base change.
\end{proof}

Let $X$ be an algebraic stack.
A point $x$ in $X$ may be represented by a morphism
$\xi\colon \mathcal{G} \to X$,
where $\mathcal{G}$ is a singleton algebraic stack.
A common choice for $\mathcal{G}$ is the spectrum of a field or the residual gerbe at the point.
For our purposes, it will be convenient to choose $\mathcal{G}$ slightly differently.

\begin{definition}
\label{def-standard-rep}
Let $f\colon X \to S$ be a morphism of algebraic stacks, and let $x$ be a point of $X$.
We call a morphism $\xi\colon \mathcal{G} \to X$ a \emph{standard representative} for $x$ over $S$
provided that it satisfies the following conditions.
\begin{enumerate}
\item
\label{def-standard-rep-i}
The stack $\mathcal{G}$ is a singleton gerbe and the unique point of $\mathcal{G}$ maps to~$x$.
\item
\label{def-standard-rep-ii}
The morphism $\xi$ is stabilizer preserving over~$S$.
\item
\label{def-standard-rep-iii}
The canonical morphism $I_{\mathcal{G}/S} \to I_{\mathcal{G}}$ is an isomorphism.
\end{enumerate}
\end{definition}

\begin{remark}
\label{rem-standard-rep}
Under the assumption that $\mathcal{G}$ is a gerbe,
Condition~\itemref{def-standard-rep-iii} is equivalent to saying that the composition $f\circ \xi\colon \mathcal{G} \to S$
factors (uniquely up to unique isomorphism) through the coarse space of $\mathcal{G}$.
\end{remark}

\begin{example}
Assume that $S$, as in Definition~\ref{def-standard-rep}, is an algebraic space.
Then the residual gerbe, as defined in~\cite[\sptag{06MU}]{stacks-project},
at a given point would be a natural choice of standard representative,
provided that it exists.
Residual gerbes are known to exist under quite general circumstances
--- for instance if $X$ is quasi-compact and quasi-separated over~$S$ \cite[\sptag{06RD}]{stacks-project}.
\end{example}

\begin{example}
In general, one can always construct a standard representative from a given morphism $x\colon \Spec k \to X$
by taking $\mathcal{G}$ as the residual gerbe of the stack $X\times_S\Spec k$ at the section induced by~$x$,
and letting $\xi\colon \mathcal{G} \to X$ be the inclusion of $\mathcal{G}$ in $X\times_S\Spec k$ composed
with the projection to~$X$.
Note that the residual gerbe always exists in this situation by~\cite[\sptag{06G3}]{stacks-project}.
Also note that in this case,
the gerbe $\mathcal{G}$ is canonically isomorphic to the the classifying stack of the relative
stabilizer $G_{x/S}$ at $x$ over~$S$.
\end{example}

\begin{lemma}
\label{lem-standard-ind}
Let $X \to S$ be a morphism of algebraic stacks and $x \in X$ a point.
Then any pair $\xi, \xi'$ of standard representatives for $x$ can be completed
to a commutative diagram
$$
\xymatrix{
& \mathcal{G}'' \ar[dl]_\eta\ar[dr]^{\eta'}&\\
\mathcal{G} \ar[dr]_\xi & & \mathcal{G}' \ar[dl]^{\xi'}\\
& X & \\
}
$$
such that $\xi\eta = \xi'\eta'$ is a standard representative for~$x$.
In particular, the morphisms $\eta$ and $\eta'$ are stabilizer preserving.
\end{lemma}
\begin{proof}
One can take $\mathcal{G}''$ as the residual gerbe at any point in~$\mathcal{G}\times_X\mathcal{G}'$.
Indeed, it is clear that $\eta$ (and $\eta'$) will be stabilizer preserving over $S$,
since this property is stable under base change and under composition with monomorphisms.
Since $I_{\mathcal{G}}\to I_S$ is trivial, so is the composition
$I_{\mathcal{G''}}\to I_{\mathcal{G}}\to I_S$ and $I_{\mathcal{G''}/S} \cong
I_{\mathcal{G''}}$.
\end{proof}

Let us briefly recall some of the basic properties of the cotangent complex for an algebraic stack.
We refer to \cite[Chapitre~17]{laumon} and \cite[Section~8]{olsson_sheaves-on-artin-stacks}
for a more complete account.
The cotangent complex $L_{X/S}$ is defined%
\footnote{Strictly speaking, the cited references only construct
the pro-object $\{\tau_{\geq -n} L_{X/S}\}_n$ and assumes that $\pi$ is quasi-compact and quasi-separated.
The first assumption can be removed using~\cite[Example~2.2.5]{laszlo-olsson_six-operations-I} and the second is apparently not used.
In any case, in our applications $\pi$ is smooth and of finite presentation
so we only need $\tau_{\geq 0} L_{X/S}$.}
for any morphism $\pi\colon X \to S$.
If $\pi$ is smooth then $L_{X/S}$ is perfect with tor-amplitude in $[0, 1]$,
and the rank of $L_{X/S}$ equals the relative dimension of $\pi$~\cite[Proposition~17.10]{laumon}.
Moreover, $L_{X/S} \cong \Omega_{X/S}[0]$ in the case when $\pi$ is smooth and
represented by a Deligne--Mumford stack~\cite[Corollaire~17.9.2]{laumon}.
If $\pi$ is a regular immersion, then $L_{X/S} \cong \mathcal{N}_{X/S}[1]$,
where $\mathcal{N}_{X/S}$ denotes the conormal bundle of the immersion~\cite[\sptag{08SK}]{stacks-project}.

For computations with the cotangent complex,
we will frequently use the \emph{fundamental triangle}~\cite[Proposition~17.3(3)]{laumon}
\begin{equation}
\Ld f^*L_{Y/S} \to L_{X/S} \to L_{X/Y} \to,
\end{equation}
induced by a morphism $f\colon X \to Y$ over $S$,
and the \emph{tor-independent base change property}~\cite[Proposition~17.3(4)]{laumon},
which asserts that for any tor-independent cartesian square
\begin{equation}
\vcenter{\xymatrix{
X' \ar[r]^g\ar[d] & X\ar[d]\\
S' \ar[r] & S,}}
\end{equation}
we have $L_{X'/S'} \cong \Ld g^*L_{X/S}$.

We are now ready to make the main definition of the section.
\begin{definition}
\label{def-codim-of-stack}
Let $X \to S$ be a morphism of algebraic stacks,
which is smooth and has finite, tame relative inertia.
The \emph{codimension of stackiness} of $X$ over $S$ at the point $x \in X$ is defined
as the relevant rank (see Definition~\ref{def-relevant-rank})
of the pull-back $\xi^*L_{X/S}$ of the cotangent complex along a standard
representative (see Definition~\ref{def-standard-rep})
$\xi\colon \mathcal{G} \to X$ for~$x$ over~$S$.\\
\emph{Note that in light of Lemma~\ref{lem-relevant-bs} and Lemma~\ref{lem-standard-ind},
the definition does not depend on the choice of standard representative~$\xi$ for~$x$.}
\end{definition}

The following lemma encapsulates the main result of Appendix~\ref{app-mult-type}
and will be used several times.

\begin{lemma}
\label{lem-relevant-gerbe}
Let $f\colon X \to S$ be a gerbe with finite tame inertia,
and let $\xi\colon \mathcal{G} \to X$ be a morphism from a connected gerbe with finite tame inertia.
Then the relevant rank of $\xi^*L_{X/S}$ vanishes.
\end{lemma}
\begin{proof}
In light of Lemma~\ref{lem-relevant-bs},
we may assume that $\mathcal{G}$ is a gerbe over $\Spec k$ for some algebraically closed
field~$k$.
Let $g\colon X_\mathcal{G}\to \mathcal{G}$ be the pull-back of $f$ along the composition
$f\circ\xi$,
and let $\sigma\colon \mathcal{G} \to X_\mathcal{G}$ denote the section induced by $f$.
By tor-independent base change,
it suffices to show that the relevant rank of $\sigma^*L_{X_\mathcal{G}/\mathcal{G}}$ vanishes.
Consider the fundamental triangle
\begin{equation}
\label{gerbe-fundamental}
g^*L_{\mathcal{G}/k} \to L_{X_\mathcal{G}/k} \to L_{X_\mathcal{G}/\mathcal{G}} \to
\end{equation}
Since $k$ is algebraically closed, the gerbes $X_\mathcal{G}$ and $\mathcal{G}$
are both isomorphic to classifying stacks of some well-split finite tame group schemes over~$k$.
In particular, the two first terms of \eqref{gerbe-fundamental} are zero in $\K_0(X_\mathcal{G})$
by Proposition~\ref{prop-well-split-gerbe}.
Hence the same must hold for the third term,
which gives the desired vanishing.
\end{proof}

Next we study the functorial properties of codimension of stackiness.

\begin{proposition}
\label{prop-cs-functorial}
Let $\pi\colon X \to S$ be a morphism of algebraic stacks that is smooth with finite and tame relative inertia.
Then the codimension of stackiness is invariant under arbitrary base change $S' \to S$ and under
smooth morphisms $Y \to X$ that are either stabilizer preserving over~$S$ or tame gerbes.
\end{proposition}
\begin{proof}
Let $X'$ be the base change of $X$ along an arbitrary morphism $S'\to S$.
Choose a standard representative $\xi\colon \mathcal{G} \to X'$ over $S'$.
Then we have a commutative diagram
$$
\xymatrix{
I_\mathcal{G} \ar[r]\ar[d] & I_{X'/S'} \ar[r]\ar[d] & I_{X/S} \ar[d] \\ 
\mathcal{G} \ar[r]^\xi\ar[d] & X' \ar[r]^g\ar[d] & X \ar[d]\\
\mathcal{G}_\cs \ar[r] & S' \ar[r] & S\\
}
$$
where the lower left hand square exists by Remark~\ref{rem-standard-rep}.
The same remark gives $I_{\mathcal{G}/S} = I_{\mathcal{G}}$.
The upper left hand square is cartesian by the assumption on $\xi$,
and the upper right hand square is cartesian since the same holds for the
lower right hand square.
It follows that the composition $g\circ\xi$ is a standard representative for the
image of $\xi$ under $g$,
so invariance follows from tor-independent base change of the cotangent complex.

Now assume that $f\colon Y \to X$ is smooth,
and let $\xi\colon \mathcal{G} \to Y$ be a standard representative for a point in~$Y$.
By invariance under base change together with Lemma~\ref{lem-relevant-bs},
we may assume that $S = \mathcal{G}_\cs = \Spec k$, where $k$ is a field. 

First assume that $f$ is stabilizer preserving.
Then it follows that $f\circ \xi$ is a standard
representative for the image of $\xi$ under $f$.
By the triangle
$$
f^*L_{X/S} \to L_{Y/S} \to L_{Y/X} \to
$$
it therefore suffices to show that the relevant rank of $\xi^*L_{Y/X}$ vanishes.
Let $\widetilde{Y}\to \mathcal{G}$ be the base change of $Y\to X$ along the map $f\circ \xi$,
and let $\sigma\colon \mathcal{G} \to \widetilde{Y}$ be the section induced by~$\xi$.
By the base change property of the cotangent complex, it is enough to show
that the relevant rank of $\sigma^* L_{\widetilde{Y}/\mathcal{G}}$ vanishes.
By our assumption, the morphism $\widetilde{Y} \to \mathcal{G}$ is stabilizer preserving.
In particular, $\widetilde{Y}$ is a gerbe, since the same holds for $\mathcal{G}$,
and the commutative square
$$
\xymatrix{
\widetilde{Y} \ar[r]\ar[d]& \widetilde{Y}_{\cs} \ar[d]\\
\mathcal{G} \ar[r] & S\\
}
$$
is cartesian.
Again by the base change property of the cotangent complex,
it suffices to show that the relevant rank of the pull-back of $L_{\widetilde{Y}_{\cs}/S}$ vanishes.
But $\widetilde{Y}_{\cs}$ is representable,
so the map $\mathcal{G} \to \widetilde{Y}_{\cs}$ factors through $S=\Spec k$,
and the result follows.

Finally, we assume instead that $f$ is a gerbe.
Then we may obtain $\xi$ as the base change
\begin{equation}
\label{eq-smooth-functorial}
\vcenter{\xymatrix{
\mathcal{G} \ar[r]^{\smash{\xi}}\ar[d]_\pi & Y \ar[d]^f\\
\mathcal{H} \ar[r]_\eta  & X\\
}}
\end{equation}
of a standard representative $\eta\colon \mathcal{H} \to X$.
Since $\pi$ is a gerbe,
it induces an epimorphism on inertia. 
In particular, pull-back along $\pi$ preserves the relevant rank.
Similarly as before, it therefore suffices to prove that the relevant rank
of $\xi^*L_{Y/X}$ vanishes,
which it does by~Lemma~\ref{lem-relevant-gerbe}.
\end{proof}

In the particular case when we are considering a global quotient stack,
the codimension of stackiness has a rather explicit description in the spirit of the
main idea described in the introduction of this section.
We review some facts about the tangent space of an equivariant algebraic space.

Let $X$ be an algebraic space over a field~$k$,
and let $G$ be a group scheme over~$k$ acting on~$X$.
Recall that $G$ and~$X$ are characterized by their point functors
$$
G(A) = \Mor_S(\Spec A, G), \qquad X(A) = \Mor_S(\Spec A, X),
$$
where $A$ is a $k$-algebra,
and that the action of $G$ on~$X$ is characterized by the induced action of $G(A)$ on $X(A)$,
which is functorial in $A$.
The \emph{tangent bundle} $T_{X/k}$ is the functor $T_{X/k}(A) := X(A[\varepsilon]/\varepsilon^2)$.
It comes with a structure morphism $T_{X/k} \to X$ which is represented by the morphism
$\mathbb{V}_X(\Omega_{X/k}) \to X$.
Note that $T_{X/k}$ is naturally endowed with a $G$-action given on $\Spec A$-points by
$$
G(A)\times T_{X/k}(A) \to T_{X/k}(A) \cong X(A[\varepsilon]/\varepsilon^2), \qquad (\gamma, \xi) \mapsto \gamma'\xi,
$$
where $\gamma'$ denotes the image of $\gamma$ under $G(A) \to G(A[\varepsilon]/\varepsilon^2)$.
Moreover, the structure morphism $T_{X/k} \to X$ is equivariant with respect to this group action.
Given a $k$-point $x \in X(k)$, we also consider the \emph{tangent space} $T_{X/k, x}$ at $x$,
which is the fiber of the morphism $T_{X/k} \to X$ at $x$.
The tangent space is endowed with a canonical action by the stabilizer $G_x$ of $G$ at $x$,
making it a $k$-linear representation of the group scheme~$G_x$.

\begin{proposition}
\label{prop-cs-quotient}
Let $X$ be a smooth algebraic space over a field~$k$,
and let $G$ be a finite tame group scheme over~$k$,
acting on~$X$.
Choose a $k$-point $x$ on $X$, and denote the stabilizer group scheme at $x$ by $G_x$.
Then the codimension of stackiness at the corresponding point of the quotient stack~$[X/G]$ over~$k$
equals the relevant rank of the tangent space $T_{X/k, x}$ of $X$ at $x$ considered as a
$k$-linear $G_x$-representation.
\end{proposition}
\begin{proof}
We have a $G$-equivariant sequence of morphisms
$$
T_{X/k} \xrightarrow{\tau} X \xrightarrow{\pi} \Spec k,
$$
which descends to a sequence
$$
[T_{X/k}/G] \cong T_{[X/G]/\B G} \xrightarrow{\tau_G} [X/G] \xrightarrow{\pi_G} \B G,
$$
of quotient stacks.
Let $\xi\colon \B G_x \to [X/G]$ be the standard representative of the point on $[X/G]$ over~$k$
induced by~$x$.
Then we have an isomorphism $[T_{X/k, x}/G_x] \cong \xi^*T_{[X/G]/\B G}$ over $\B G_x$.
In particular, since $T_{[X/G]/\B G}$ is represented by $\mathbb{V}_{[X/G]}(\Omega_{[X/G]/\B G})$,
the relevant rank of the tangent space equals the relevant rank of $\xi^*\Omega_{[X/G]/\B G}$,
which is quasi-isomorphic to $\xi^*L_{[X/G]/\B G}$.
Now the composition $[X/G] \xrightarrow{\pi_G} \B G \to \Spec k$ gives a triangle
$$
\pi_G^*L_{\B G/k} \to L_{[X/G]/k} \to L_{[X/G]/\B G} \to.
$$
By Lemma~\ref{lem-relevant-gerbe},
the relevant rank of $\xi^*\pi_G^*L_{\B G/k}$ vanishes.
Hence the relevant ranks of $\xi^*L_{[X/G]/k}$ and $\xi^*L_{[X/G]/\B G}$ are equal,
concluding the proof.
\end{proof}

In fact, we will only need to use Proposition~\ref{prop-cs-quotient} in the following concrete example.
\begin{example}
\label{ex-tangent-space-linear}
Let $G$ be a finite tame group scheme over a field~$k$,
and let $V$ be an $n$-dimensional $G$-representation.
Then we get a $G$-scheme $X = \mathbb{V}_{k}(V)$,
whose $k$-points correspond to elements in the dual space $V^\vee$.
The tangent space at a point corresponding to an element $v \in V^\vee$
is isomorphic to the restriction of the representation $V^\vee$ along
$G_v \to G$, where $G_v$ denotes the stabilizer of $v$.

Let $V^\vee = V^\vee_\triv \oplus V^\vee_\nt$ be the canonical splitting
into the trivial and non-trivial parts of the $G$-representation~$V^\vee$.
From Proposition~\ref{prop-cs-quotient},
we see that the codimension of stackiness at any $k$-point of $[X/G]$
lies in the interval $[0, m]$, where $m = \dim_k V^\vee_\nt$.
Furthermore, the value~$m$ is obtained if and only if~$v \in V^\vee_\triv$.
\end{example}

Finally, we combine the explicit description in Example~\ref{ex-tangent-space-linear}
with the functorial properties from Proposition~\ref{prop-cs-functorial} and the local structure
theorem for tame stacks in Theorem~\ref{T:Luna-slice-finite}
and obtain the following result.

\begin{theorem}
\label{thm-cs-properties}
Let $\pi\colon X \to S$ be a morphism of algebraic stacks which is smooth of finite presentation with finite and tame relative inertia.
Then the codimension of stackiness of $X$ relative $S$ is an upper semi-continuous function taking its values in the
interval~$[0, \dim X/S]$.

Assume that the codimension of stackiness has a maximal value of $m$ on $X$.
Then the locus where this maximal value is obtained has a unique structure of a closed substack $Z \subset X$
which is smooth over~$S$.
Moreover, the substack $Z$ has the following properties:
\begin{enumerate}
\item
it has codimension $m$ in~$X$;
\item\label{TI:cs-properties:gerbe}
it is a gerbe over its relative coarse space over~$S$;
\item
it has normal crossings with any simple normal crossings divisor on~$X$
(Definition~\ref{def-snc}).
\end{enumerate}
\end{theorem}
\begin{proof}
The questions is Zariski-local on $X$ so we may assume that $X$ has relative dimension~$n$
over~$S$.

First we note that the question is fppf-local on $S$ in the following sense.
Let $S' \to S$ and $S''\to S'\times_S S'$ be faithfully flat morphisms locally
of finite presentation. Then the theorem for $S'$ and $S''$ implies the theorem
for $S$. Indeed, we have an induced fppf covering $X':=X\times_S S'\to X$,
which respects the codimension of stackiness by the base change part of Proposition~\ref{prop-cs-functorial}.
Since flat maps are open, this gives us upper semi-continuity of the codimension of stackiness on~$X$.
Let $Z'\subset X'$ be the unique closed substack structure of the maximal locus
that is smooth over $S'$. Uniqueness of $Z'$
implies that there is at most one smooth structure $Z$ over $S$.
Moreover, uniqueness over $S''$ implies that the two pull-backs of $Z'$ along
$S'\times_S S' \rightrightarrows S'$ agree. By descent, we thus deduce the
existence of $Z$ with its desired properties.
In particular, we may reduce to the case where $S$ is affine.
Similarly, the question is local on $X$ for smooth stabilizer preserving
morphisms.

Secondly, we reduce to the case when $S = \Spec R$ is noetherian.
Since $X$ is of finite presentation over $S$,
we have a cartesian diagram
\begin{equation}
\label{eq-cs-us-second}
\vcenter{\xymatrix{
X \ar[r]\ar[d] & X_0\ar[d] \\
S \ar[r] & S_0,
}}
\end{equation}
where $S_0$ is affine and noetherian and $X_0$ over $S_0$ satisfies the
hypothesis of the proposition. Similarly, given a simple normal crossings
divisor on $X$, we may assume that it is a pull-back from a simple normal
crossings divisor on $X_0$.
Again, since the codimension of stackiness respects arbitrary base change,
the theorem for $X_0\to S_0$ gives
upper semi-continuity on $X$ as well as the existence of $Z \subset X$.
Furthermore, it is clear that $Z$ has all the desired properties except possibly for uniqueness.
If $\widetilde{Z}\subset X$ is another closed substack, smooth over $S$,
then $\widetilde{Z}\subset X$ is of finite presentation so
we may choose a suitable diagram as in \eqref{eq-cs-us-second}
such that both $Z$ and $\widetilde{Z}$ are defined over $S_0$.
This gives uniqueness of $Z$.

Thirdly, we reduce to the situation where there exists a diagram
\begin{equation}
\label{eq-cs-us-third}
\vcenter{\xymatrix{
X \ar[dr]\ar[rr] && \mathbb{V}_{\B G}(\mathcal{E})\ar[dl]\\
& S, &
}}
\end{equation}
where $G$ is a finite locally free tame group scheme on $S$,
$\mathcal{E}$
a rank~$n$ locally free sheaf on $\B G$
and the horizontal arrow is \'etale and stabilizer preserving.
Furthermore, we assume that $\mathcal{E}$ decomposes as
$\mathcal{E}^1 \oplus \cdots \oplus \mathcal{E}^r \oplus \mathcal{E}'$,
with each $\mathcal{E}^i$ of rank~1 corresponding to a component of the
given simple normal crossings divisor on~$X$.
Pick a closed point $x\in |X|$ with image $s\in |S|$. Then the residue
field extension $\kappa(x)/\kappa(s)$ is finite. We can thus find a
(quasi-finite) flat morphism $S'\to S$ and points $x'\in |X\times_S S'|$ and
$s'\in |S'|$ above $x$ and $s$ such that $\kappa(x')=\kappa(s')$
\cite[Corollaire~0.10.3.2]{egaIII}.
Hence we may use the local structure theorem for smooth tame stacks (Theorem~\ref{T:Luna-slice-finite})
to reduce to the situation in diagram~\eqref{eq-cs-us-third},
after first replacing $S$ by a flat neighborhood of $s$ and then replacing
$X$ by a stabilizer preserving \'etale neighborhood of $x$.

Now we assume that we are in the situation obtained after the third reduction.
By tameness of $G$,
the bundle $\mathcal{E}$ decomposes canonically as $\mathcal{E}_\triv\oplus\mathcal{E}_\nt$
into a trivial and a non-trivial part (Theorem~\ref{T:tame-splitting:QCoh}).
This decomposition is preserved under base change to any fiber of $S$.
In particular, it follows from the explicit computation in Example~\ref{ex-tangent-space-linear}
that the the linear substack $\mathbb{V}_{\B G}(\mathcal{E}_\triv)\subset
\mathbb{V}_{\B G}(\mathcal{E})$ gives a smooth stack structure of the locus
where the codimension of stackiness
obtains its maximum.
Note that $\mathbb{V}_{\B G}(\mathcal{E}_\triv)$ is a trivial gerbe over
$\mathbb{V}_S(\mathcal{E}_\triv)$
and that its codimension is equal to $\rk \mathcal{E}_\nt$ --- the codimension of stackiness.
Moreover, the non-trivial part $\mathcal{E}_\nt$ is clearly the direct sum of the non-trivial parts
of $\mathcal{E}^i$ and $\mathcal{E}'$,
so the subspace has normal crossings with the given simple normal crossings divisor.
Since the codimension of stackiness respects composition with smooth stabilizer preserving maps,
we obtain $Z\subset X$ as the pull-back of this linear substack to~$X$.
Again all the desired properties for $Z$, except for uniqueness, follows directly.

It only remains to prove that $Z$ is unique.
This is obvious provided that $S = \Spec R$ is reduced.
Otherwise, let $I \subset R$ denote the ideal of nilpotents and define $R_i = R/I^{i + 1}$.
By the noetherian assumption, we get a finite factorization $R = R_j \to \cdots \to R_0 = R/I$
of the reduction homomorphism,
with each $R_{i} \to R_{i-1}$ being the quotient by the square zero ideal~$I^i/I^{i+1}$.
Denote the base change of $X$ to $\Spec R_i$ by~$X_i$.
Let $Z_0 \hookrightarrow X_0$ denote the smooth closed substack where the codimension
of stackiness is maximal. 
We need to show that for each~$i$,
there is an essentially unique cartesian diagram
\begin{equation}
\label{eq-codim-lift}
\vcenter{\xymatrix{
Z_{i-1} \ar[r]\ar[d] & Z_{i} \ar[d]\\
X_{i-i} \ar[r] & X_{i} \\
}}
\end{equation}
such that $Z_{i}$ is smooth over $R_{i}$.
This lifting problem is controlled by
\begin{equation}
\label{eq-exts}
\Ext^j_{\mathcal{O}_{Z_0}}(\mathcal{N}_{Z_0/X_0}, \mathcal{O}_{Z_0}\otimes_{R_0}I^i/I^{i+1}),
\end{equation}
where $\mathcal{N}_{Z_0/X_0}$ denotes the conormal bundle associated to the inclusion
$Z_0 \hookrightarrow X_0$. This follows from a slight variation of~\cite[Theorem~1.4]{olsson_def-theory-of-stacks}:
instead of assuming that the morphism
$Z_0\to X_0$ is flat, it is sufficient that $Z_0$ and $X_0$ are flat over a
common base $S_0$ and that $X_i\to S_i$ are flat deformations of $X_0\to S_0$.
More specifically, the obstructions to finding lifts as in \eqref{eq-codim-lift} live
in~$\Ext^1$,
and the set of liftings form a torsor under $\Ext^0$
provided that the lifting problem is unobstructed.

The ext-groups \eqref{eq-exts} are given as the cohomology groups of the complex
$$
\Rd\Hom_{\mathcal{O}_{Z_0}}(\mathcal{N}_{Z_0/X_0}, \mathcal{O}_{Z_0}\otimes_{R_0}I^i/I^{i+1}),
$$
which is isomorphic to 
$\Rd\Gamma(Z_0, \mathcal{N}_{Z_0/X_0}^\vee) \otimes^{\Ld}_{R_0}I^i/I^{i+1}$,
by dualization combined with the projection formula.
Let $\pi\colon Z_0 \to (Z_0)_\cs$ denote the canonical projection to the coarse space.
Since $Z_0$ is tame, the higher direct images of $\pi_*$ vanish,
so it suffices to show that the sheaf $\pi_*\mathcal{N}_{Z_0/X_0}^\vee$ vanishes.
But this follows from the fact that $\mathcal{N}_{Z_0/X_0}$ is the pull-back of the sheaf
$\mathcal{E}_\nt$ on $\B G$ along a stabilizer preserving morphism,
and that the trivial part of $\mathcal{E}_\nt$ is zero,
which concludes the proof.
\end{proof}


\end{section}


\begin{section}{Divisorialification}\label{S:divisorialification}
In this section,
we generalize the divisorialification algorithm in
\cite[Algorithm~C]{bergh_functorial-destackification}
to algebraic stacks with tame stabilizers that are not necessarily abelian.

Let $S$ be an algebraic stack,
and let $(X, D)$ be a standard pair over~$S$ (see Definition~\ref{def-standard-pair}).
Composing the morphism $\delta_{D/S} \colon X \to [\mathbb{A}^n_S/\Gm^n]$ in
Remark~\ref{rem-snc-morphism} with
the canonical morphism $[\mathbb{A}^n_S/\Gm^n] \to \B\Gm^n \times S$ yields a morphism
\begin{equation}
\label{eq-snc-torsor}
\tau_{D/S}\colon X \to \B \Gm^n \times S.
\end{equation}
The corresponding morphism of inertia stacks induces a group homomorphism 
\begin{equation}
\label{eq-divisorial-rep}
\rho_{D/S}\colon I_{X/S}\to \Gm^n \times X
\end{equation}
over~$X$.
Here we have used that $\Gm^n$ is abelian,
which implies that the inertia stack of $\B \Gm^n$ can be identified with $\Gm^n \times \B \Gm^n$.

\begin{definition}
\label{def-divisorial-index}
Let $S$ be an algebraic stack and let $(X, D)$ be a standard pair over $S$.
Assume that the relative inertia $I_{X/S}$ is finite and tame.
The \emph{divisorial index} of the pair $(X, D)$ over $S$ at a point $x \in X$
is defined as the codimension of stackiness (see Definition~\ref{def-codim-of-stack}) of
the morphism \eqref{eq-snc-torsor} at $x$.
We say that the pair $(X, D)$ is \emph{divisorial} provided that the divisorial index
is everywhere zero.
\end{definition}

\begin{remark}
\label{rem-di-empty-ordering}
The divisorial index only depends on the underlying divisor of~$D$.
In particular, the ordering of the components of~$D$ is irrelevant.
\end{remark}

\begin{remark}
\label{rem-divisorial}
It is not hard to see that the definition of divisorial index given in Definition~\ref{def-divisorial-index}
is equivalent to the one given in \cite[Definition~7.6]{bergh_functorial-destackification} provided that
$S$ is a scheme and $X$ has diagonalizable stabilizers.
\end{remark}

\begin{proposition}
\label{prop-divisorial}
Let $S$ be an algebraic stack and let $(X, D)$ be a standard pair over $S$.
Assume that the relative inertia $I_{X/S}$ is finite and tame and that
$(X, D)$ is divisorial over $S$.
Then we have a total rigidification $\pi\colon (X, D) \to (X_{\rig/S}, D_{\rig/S})$ relative~$S$
in the sense of Section~\ref{S:sp-rigidification}.
Moreover, the following statements hold:
\begin{enumerate}
\item
\label{it-p-div-coarse}
$\tau_{D/S}$ factors as $\tau_{D_{\rig/S}/S}\circ \pi$ identifying $\tau_{D_{\rig/S}/S}$
with the relative coarse space of $\tau_{D/S}$;
\item
\label{it-p-div-representable}
$X_{\rig/S} \to S$ is representable away from $D_{\rig/S}$;
\item
\label{it-p-div-representation}
$\rho_{D_{\rig/S}/S}\colon I_{X_{\rig/S}/S}\inj \Gm^n\times X_{\rig/S}$ is
a monomorphism.
\item
\label{it-p-div-divisorial}
$(X_{\rig/S}, D_{\rig/S})$ is a divisorial standard pair over $S$;
\end{enumerate}
In particular, the stack $X_{\rig/S}$ has diagonalizable relative stabilizers over $S$.
\end{proposition}
\begin{proof}
By assumption, the codimension of stackiness of $X$ relative $\B \Gm^n\times S$ is identically zero.
Hence $X$ is a gerbe over its relative coarse space over $\B \Gm^n\times S$ by Theorem~\ref{thm-cs-properties}~\itemref{TI:cs-properties:gerbe}.
We denote this gerbe, which is clearly tame, by $\pi\colon X\to X_{\rig/S}$.
The identity $\pi^{-1}(D_{\rig/S}) = D$ gives \ref{it-p-div-coarse}.
Since $\tau_{D_{\rig/S}/S}$ is representable,
the same holds, \emph{a fortiori},
for the morphism $\delta_{D_{\rig/S}/S}\colon X_{\rig/S} \to [\mathbb{A}^n_S/\Gm^n]$
corresponding to $D_{\rig/S}$,
which implies \ref{it-p-div-representable}.
In particular, since the complement of $D_{\rig/S}$ in $X_{\rig/S}$ is schematically dense,
the rigidification $\pi$ is indeed total.
The kernel of $\rho_{D_{\rig/S}/S}$ is the relative inertia $I_{X_{\rig/S}/\B \Gm^n \times S}$.
This vanishes since $\tau_{D_{\rig/S}/S}$ is representable, which gives \ref{it-p-div-representation}.
By Proposition~\ref{prop-cs-functorial},
the codimension of stackiness is invariant under tame gerbes,
which gives \ref{it-p-div-divisorial}.
\end{proof}


\begin{proposition}
\label{prop-di-blow-up}
Let $S$ be an algebraic stack, and let $(X, D)$ be a standard pair over $S$ (see Definition~\ref{def-standard-pair}).
\begin{enumerate}
\item\label{PI:di:bound}
  The divisorial index at any point in $X$ is a natural number bounded by the relative dimension
of $X$ over~$S$.
\item\label{PI:di:usc}
  The divisorial index with respect to $(X, D)$ is an upper semi-continuous function on~$|X|$.
\item\label{PI:di:center}
  If the divisorial index obtains a maximum~$m$,
then the locus where this maximum is obtained admits a unique structure $Z \subset X$
of a closed substack that is smooth over~$S$. The closed substack $Z$ has codimension $m$ and normal crossings with $D$.
Furthermore, the divisorial index of the transform $(\widetilde{X}, \widetilde{D})$ of the blow-up of
$(X, D)$ in $Z$ is everywhere strictly smaller than~$m$.
\end{enumerate}
\end{proposition}
\begin{proof}
First we prove \itemref{PI:di:bound}.
Let $G$ denote the torus $\Gm^n \times S$ over~$S$.
Consider the fundamental triangle
\begin{equation}
\label{eq-fundamental-div}
\tau_D^*L_{\B G/ S} \to L_{X/S} \to L_{X/\B G} \to
\end{equation}
associated to the composition $X \to \B G \to S$.
Given a point $x \in X$, we choose a sequence of morphisms $\mathcal{G} \to \mathcal{G}' \to X$
such that the second morphism $\xi'$ is a standard representative for $x$ relative~$S$
and the composition $\xi$ is a standard representative for $x$ relative $\B G$.
The divisorial index is the relevant rank of $\xi^*L_{X/\B G}$.
This is equal to the relevant rank of $\xi^*L_{X/S}$ since the first term
in \eqref{eq-fundamental-div} is defined over $\B G$.
This is in turn bounded above by the relevant rank of $(\xi')^*L_{X/S}$,
which is the codimension of stackiness of $X$ at $x$ relative~$S$.
In particular, this number is bounded by the relative dimension of $X$ over $S$
by Theorem~\ref{thm-cs-properties}.

For the rest of the statements of the proposition,
we may replace the morphism $X \to S$ with $\tau_{D/S}\colon X \to \B G$
and assume that $D = \emptyset$.
In particular, \itemref{PI:di:usc} and the first part of \itemref{PI:di:center}
is a direct consequence of the
main properties of the codimension of stackiness described in Theorem~\ref{thm-cs-properties}.
It remains to prove that the divisorial index decreases after the blow-up.
This question is local on $S$, so we may assume that~$S$ is a scheme.
Consider the diagram
\begin{equation}
\label{eq-div-blow-up}
\vcenter{\xymatrix{
E \ar[r]\ar[d] & \widetilde{X} \ar[d]\ar[dr]^{\tau_E} &\\
Z \ar[r]\ar[d] & X \ar[d] & \B \Gm\times S \ar[dl]\\
Z_\cs \ar[r] & S &
}}
\end{equation}
where $E$ is the exceptional divisor of the blow-up,
$Z_\cs$ is the coarse space of $Z$,
and $\tau_E$ is the canonical map corresponding to the line bundle~$\mathcal{O}_{\widetilde{X}}(E)$.
Let $x \in E$ be a point.
It suffices to prove that the codimension of stackiness of $X$ at $x$ relative $\B \Gm\times S$ is strictly
smaller than $m$.

Let $\xi\colon \mathcal{G}\to E$ be a standard representative of~$x$ relative $\B \Gm\times S$,
and view \eqref{eq-div-blow-up} as a diagram in the undercategory of $\mathcal{G}$.
Given a perfect complex $\mathcal{P}$ on any of the stacks in the diagram \eqref{eq-div-blow-up},
we denote the relevant rank of its pull-back to $\mathcal{G}$ by $\{\mathcal{P}\}$.
We need to prove that $\{L_{\widetilde{X}/\B \Gm\times S}\} < m$.
By the fundamental triangle for the cotangent complex, we have
$$
\{L_{\widetilde{X}/\B \Gm\times S}\} = \{L_{\widetilde{X}/S}\} - \{L_{\B \Gm\times S/S}\} = \{L_{\widetilde{X}/S}\}
$$
where the last equality follows since $L_{\B \Gm\times S/S}$ is defined over $\B \Gm\times S$.
Using the fundamental triangle again,
we get 
$$
\{L_{E/S}\} = \{L_{E/\widetilde{X}}\} + \{L_{\widetilde{X}/S}\} = \{L_{\widetilde{X}/S}\},
$$
where the last equality follows from the fact that $L_{E/\widetilde{X}}$ is the pull-back of $\mathcal{O}_{\widetilde{X}}(-E)$
concentrated in degree $-1$, and that this line bundle is defined over $\B \Gm\times S$.
Similarly, we have
\begin{equation}
\label{eq-es-sum}
\{L_{E/S}\} = \{L_{E/Z}\} + \{L_{Z/Z_\cs}\} + \{L_{Z_\cs/S}\}.
\end{equation}
Here $\{L_{Z/Z_\cs}\}$ vanishes by Lemma~\ref{lem-relevant-gerbe} since $Z$ is a gerbe over $Z_\cs$,
and $\{L_{Z_\cs/S}\}$ vanishes since $Z_\cs$ is representable.
It follows that the codimension of stackiness at $x$ equals $\{L_{E/Z}\}$,
which is strictly smaller than $m$ since $L_{E/Z} \cong \Omega_{E/Z}[0]$ and
$\Omega_{E/Z}$ is locally free of rank $m-1$.
\end{proof}

\begin{theorem}[Divisorialification]\label{T:divisorialification}
Let $S$ be an arbitrary algebraic stack and let $(X, D)$ be a standard pair
(see Definition~\ref{def-standard-pair}).
Assume that $X$ has relative dimension $\leq d$ over~$S$
and that the relative inertia $I_{X/S}$ is finite and tame.
Then there exists a smooth ordinary blow-up sequence
\[
\Pi\colon (X_n, D_n) \to \cdots \to (X_0, D_0) = (X, D)
\]
with $n \leq d$, such that $(X_n, D_n)$ is divisorial.

The construction is functorial with respect to arbitrary base change~$S' \to S$
and with respect to smooth morphisms $X' \to X$ that are either stabilizer preserving over~$S$ or tame gerbes.
\end{theorem}
\begin{proof}
We construct the blow-up sequence by repeatedly blowing up the locus where
the divisorial index is maximal.
It follows from Proposition~\ref{prop-di-blow-up} that each center is smooth and
has normal crossings with the given divisor.
It also follows that the procedure stops after at most~$d$ steps.
\end{proof}

\begin{remark}
Assume that $V \subset X$ is an open substack that has flat inertia relative to~$S$.
Then the divisorial index vanishes at $V$,
so $\Pi$ is an isomorphism over~$V$.
\end{remark}

\begin{remark}
In the particular case when $S$ is an algebraic space and $V \subset X$ is an open dense substack
that is an algebraic space,
then $\Pi$ is an isomorphism over~$V$ by the previous remark.
Moreover, the stacky locus of $X_n$ is contained in $D_n$,
and the geometric stabilizers of $X_n$ are diagonalizable by Proposition~\ref{prop-divisorial}.
In particular, the divisorialification algorithm is also an \emph{abelianization} algorithm in this case. 
\end{remark}

\end{section}


\begin{section}{Functorial destackification}\label{S:destackification}

We formulate a relative version of the destackification theorem.
The version stated in the introduction as Theorem~\ref{TM:destackification} is obtained
by letting $T = S$ be an algebraic space.

\begin{theorem}[Destackification]\label{T:destackification}
Let $S$ be a quasi-compact algebraic stack,
and let $(X, D)$ be a standard pair
(see Definition~\ref{def-standard-pair}) over~$S$.
Furthermore, assume that $X \to T$ is a morphism of algebraic stacks
over~$S$ such that the relative inertia $I_{X/T}$ is finite with tame fibers.
Then there exists a smooth stacky blow-up sequence
\[
\Pi\colon (Y, E) = (X_n, D_n) \to \cdots \to (X_0, D_0) = (X, D)
\]
such that the coarse space $(Y_{\cs/T}, E_{\cs/T})$ relative to~$T$
is a standard pair over~$S$.
Moreover, the stack $Y$ admits a rigidification $Y \to Y_{\rig/T}$ over~$T$
such that the canonical morphism $Y_{\rig/T} \to Y_{\cs/T}$
is an iterated root stack in $E_{\cs/T}$.

The construction is functorial with respect to arbitrary base change $S' \to S$ and
with respect to smooth morphisms~$X' \to X$ that are either stabilizer preserving over~$T$ or tame gerbes.
\end{theorem}

\begin{proof}
Assume first that $S = T$ is a scheme.
First we apply the construction in Theorem~\ref{T:divisorialification} and obtain
an ordinary blow-up sequence $(X',D')\to (X,D)$ with divisorial source.
By Proposition~\ref{prop-divisorial} the rigidification
$(X'_\rig,D'_\rig)$ is divisorial with diagonalizable stabilizers.
Hence we may apply Algorithm~E from \cite{bergh_functorial-destackification} to
$(X'_\rig,D'_\rig)$
to achieve destackification, i.e., a stacky blow-up sequence with the desired properties.

Next we assume that $S$ is a quasi-compact and quasi-separated scheme and let $T$ be arbitrary.
Without loss of generality we can replace $T$ by the relative coarse space $X_{\cs/T}$.
In particular, we may assume that $X \to T$ is separated and surjective,
and therefore $T$ is quasi-compact and quasi-separated.
Present $T$ by a smooth groupoid of affine schemes $T'' \rightrightarrows T'$.
By pull-back along $X \to T$ this induces a pair of morphisms $(X'', D'') \rightrightarrows (X', D')$
over~$X$.
Next we apply destackification to $(X', D')$ and obtain a stacky blow-up sequence starting at $X'$.
Both morphisms $X'' \rightrightarrows X'$ are smooth and stabilizer preserving,
being pull-backs from smooth morphisms of algebraic spaces.
It follows from functoriality that the pull-backs of the blow-up sequence along the two morphisms
$X'' \rightrightarrows X'$ coincide.
By the argument given in the last two paragraphs of the proof of Theorem~3.5
in \cite{bergh_factorization},
the blow-up sequence descends to a blow-up sequence on~$X'$ with the desired properties.

Next, we let $S$ be an arbitrary quasi-compact algebraic stack.
Choose a presentation $S'' \rightrightarrows S'$ with $S'$ a quasi-compact
and quasi-separated scheme,
and let $(X'', D'') \rightrightarrows (X', D')$ be the induced pair of morphisms obtained
by pulling back $(X, D)$ along $S'' \rightrightarrows S'$.
Apply destackification to $(X', D')$.
By the same argument as in the previous paragraph,
we are done once we have checked that the two pull-backs of the blow-up sequence
to $(X'', D'')$ coincide.
But this can be checked locally.
Hence it follows from functoriality since we can cover $S''$ by morphisms
$U_\alpha'' \to S''$ from quasi-compact and quasi-separated schemes such that
each composition $U_\alpha'' \to S'$ is surjective.
\end{proof}

\begin{remark}
Harper uses the same idea in \cite{harper_factorization} to deduce the relative
case from the absolute case,
but we avoid the descent argument using higher stacks by noting that we
only need to descend the centers of each blow-up at each step.
We could also have obtained the same result by consistently working with relative stabilizers
instead of absolute stabilizers for the invariants used in \cite{bergh_functorial-destackification}.
\end{remark}

\begin{remark}
The assumption in Theorem~\ref{T:destackification} that $S$ be quasi-compact
is only needed to ensure that the resulting blow-up sequence becomes finite.
If we remove this assumption, we may get an infinite blow-up sequence,
but this is a minor problem as it is still locally finite over~$S$.
We can also ensure that the blow-up sequence becomes finite by other means.
For instance, it suffices to require that $X$ has bounded dimension over $S$
and that the degrees of the fibers of $I_{X/T}$ are globally bounded.
Indeed, this ensures that there are only finitely many possible \emph{toric types}
(see \cite[Definition~6.2]{bergh_functorial-destackification}) occurring in the destackification process,
and a simple analysis of Algorithm~E in \cite{bergh_functorial-destackification} reveals
that this implies the desired finiteness.
\end{remark}

\end{section}


\begin{section}{Stabilizers of positive dimension}\label{S:infinite-stabilizers}
For the proof of Theorem~\ref{TM:unipotent-stabs}, we will need a variant
of Luna's \'etale slice theorem. A more general version, where $G$ need not be finite, is the
main theorem of~\cite{alper-hall-rydh_luna-stacks}. Since the proof is much
simpler when $G$ is finite, we include it here. When $X$ has finite
tame abelian inertia, the result
is~\cite[Proposition~5.3]{bergh_functorial-destackification}. When $X$ has finite
tame inertia, the result is Theorem~\ref{T:Luna-slice-finite}.

\begin{theorem}\label{T:Luna-slice}
Let $X$ be a smooth stack over a field $k$. Let $x\colon \Spec k\to X$ be a
closed point and let $G\subseteq \stab(x)$ be a finite linearly
reductive group scheme over $k$. Then there exists a commutative
diagram
\[
\xymatrix{%
([\A{m}_k/G],0) & (W,w)\ar[l]_-q\ar[r]^-{p} & (X,x)
}
\]
where $G$ acts linearly on $\A{m}_k$, $q$ is \'etale and stabilizer preserving, $p$ is flat,
$p^{-1}(\mathcal{G}_x)=\mathcal{G}_w$ and
the induced map $G=\stab(w)\to \stab(x)$ is the given inclusion. It can further
be arranged so that:
\begin{enumerate}
\item\label{TI:luna:smooth}
  if $\stab(x)/G$ is smooth (e.g., $\kar k=0$), then $p$ is smooth;
\item\label{TI:luna:rep}
  if $I_X\to X$ is separated, then $p$ is representable; and
\item\label{TI:luna:snc}
  if $D=\sum_i D^i$ is a divisor on $X$ with simple normal crossings such that
  $x\in D^i$ for all $i$, then $E^i=\{x_i=0\}\subset \A{m}_k$ is $G$-invariant,
  $q^{-1}(E^i)=p^{-1}(D^i)$ and the $G$-action on $x_i$ coincides with the
  $G$-action on $x^*\sO(-D^i)$;
\end{enumerate}
\end{theorem}
\begin{proof}
Let $\im\subseteq \sO_X$ denote the maximal ideal of $x\in |X|$. The group
$G$ acts on the normal space $N=\mathbb{V}(x^*\im/\im^2)=\A{m}_k=\Spec B$
where $B=\Sym(x^*\im/\im^2)$. Let $N_n=\Spec B/\im_B^{n+1}$ be the
$n$th infinitesimal neighborhood of the origin and let $X_n$ be the $n$th
infinitesimal neighborhood of $x$ in $X$, that is, the closed substack
defined by $\im^{n+1}$. Let $Y=[N / G]$ and $Y_n=[N_n / G]$.
The inclusion $G\subseteq \stab(x)$ induces a flat
representable morphism $f_0\colon \B G\to \B \stab(x)$, that is, a morphism
$f_0\colon Y_0\to X_0$. We will first extend $f_0$
to compatible flat maps $f_n\colon Y_n\to X_n$, then to a map
$\widehat{f}\colon \widehat{Y}\to X$ and finally use Artin approximation to
obtain a diagram $Y \leftarrow W\rightarrow X$ as in the statement.

We begin by extending $f_0$ to a flat morphism $f_1\colon Y_1\to X_1$.
The morphism $f_0\colon Y_0\to X_0$
deforms to a flat representable morphism $f'_1\colon Y'_1\to X_1$. Indeed, by
\cite[Theorem~1.4]{olsson_def-theory-of-stacks} the
obstruction to a flat deformation lies in $\Ext^2_{\B
  G}(L_{f_0},f_0^*\im/\im^2)$, which vanishes since $L_{f_0}$ has tor-amplitude in
$[-1,0]$ and $\B G$ is cohomologically
affine.  We note that $Y_0\inj
Y'_1$ is a trivial deformation. Indeed, by \cite[Theorem~1.1]{olsson_def-theory-of-stacks},
we have $\Exal_k(\B G,f_0^*\im/\im^2)=\Ext^1_{\B G}(L_{\B G/k},f_0^*\im/\im^2)$,
which vanishes since
$L_{\B G/k}$ has tor-amplitude in
$[0,1]$.
Thus, $Y_1\cong Y'_1$.

We proceed by extending $f_1\colon Y_1\to X_1$ to
compatible maps $f_n\colon Y_n\to X_n$ for every $n$.
By \cite[Theorem~1.5]{olsson_def-theory-of-stacks}, the
obstruction to lifting $f_{n-1}\colon Y_{n-1}\to X$ to $f_n\colon Y_n\to X$
lies in $\Ext^1_{\B
  G}(f_0^*L_{X/k},f_0^*\im^n/\im^{n+1})$, which vanishes since $L_{X/k}$ has
tor-amplitude in $[0,1]$. Since
$f_1^{-1}(X_0)=Y_0$, by the construction of $f_1$, it follows that
$f_n^{-1}(X_0)=Y_0$ and that $f_n$ factors through $X_n$. It follows that
$f_n\colon Y_n\to X_n$ is flat by the local criterion
of flatness~\cite[0.10.2.1]{egaIII}.

Next, we extend the $f_n$'s to a map $\widehat{f}\colon \widehat{Y}\to X$.
Let $\widehat{Y}=[\widehat{N} / G]=\varinjlim_n Y_n$ be
the completion, that is, $\widehat{N}=\varinjlim_n N_n=\Spec \widehat{B}$ where
$\widehat{B}=\varprojlim_n B/\im_B^{n+1}$. We also have the descriptions
$\widehat{Y}=Y\times_{Y_\cs} \widehat{Y_\cs}$ and
$\widehat{N}=N\times_{Y_\cs} \widehat{Y_\cs}$. Indeed, $Y_\cs=\Spec B^G$ and
since $B$ is a finite $B^G$-algebra,
we have $\widehat{B}=B\otimes_{B^G} \widehat{B^G}$.
Since
\[
X(\widehat{N}) = \varprojlim_n X(N_n)\quad\text{and}\quad
X(G\times \widehat{N}) = \varprojlim_n X(G\times N_n)
\]
it follows by descent that the system
$\{f_n\}$ extends to an essentially unique morphism $\widehat{f}\colon
\widehat{Y}\to X$.

The final step is to construct the diagram $Y \leftarrow W\rightarrow X$ using
Artin approximation. Thus consider the contravariant functor
\[
F\colon (T\to Y_\cs)\longmapsto \pi_0\bigl(\Mor_k(Y\times_{Y_\cs} T,X)\bigr)
\]
from affine schemes over $Y_\cs$ to sets. The functor $F$ preserves filtered colimits
since $X/k$ is locally of finite presentation.
Artin
approximation~\cite[Theorem~1.12]{artin_alg_approximation}
applied to $F$ and $\widehat{f}\in F(\widehat{Y}_\cs)$ thus gives an
\'etale neighborhood $q_\cs\colon (T,t)\to (Y_\cs,0)$ and a morphism
$p\colon W:=Y\times_{Y_\cs} T\to X$ such that there exists a $2$-isomorphism
$p|_{W_1} \cong f_1\circ q_1$.
Here $q\colon W\to Y$
is the induced stabilizer preserving \'etale morphism
and $q_n\colon W_n\to Y_n$ is the induced isomorphism on $n$th infinitesimal
neighborhoods. Note that $p|_{W_n}$ is not necessarily $2$-isomorphic to
$f_n\circ q_n$ for $n>1$.
Nevertheless, the $2$-isomorphism $p|_{W_1} \cong f_1\circ q_1$ and the
identification of conormal bundles $\sN_{W_0/W}=f_0^*\im/\im^2$ is enough
to deduce that $p^{-1}(X_0)=W_0$ and that $p$ is flat at $W_0$ by the local
criterion of
flatness~\cite[0.10.2.1]{egaIII}.

\medskip
We now consider \itemref{TI:luna:smooth} and \itemref{TI:luna:rep}.
If $\stab(x)/G$ is smooth, then $f_0$ is smooth so $p$ is smooth at $w$.
If $I_X$ is separated, then $I_{W/X}\subset I_W$ is a closed subgroup, hence
finite. Since $I_{W/X}\to W$ is trivial at $w$, it is trivial in a neighborhood of $w$. Thus, \itemref{TI:luna:smooth} and \itemref{TI:luna:rep} hold after replacing $W$ with an open neighborhood of
$w$.

For \itemref{TI:luna:snc}, we get a splitting
\[
\im/\im^2=\sN_{D_1/X,x}\oplus \dots \oplus \sN_{D_r/X,x} \oplus \sE'.
\]
By the construction above, we obtain that $p^{-1}(D^i)=q^{-1}(E^i)$ after
restricting to $W_1$. To get an equality over $W$, we modify the step where
we extend $f_{n-1}$ to $f_n$. Instead of deforming over $k$, we deform over
the stack $[\A{r}_k/\Gm^r]$ where the structure morphisms $Y\to [\A{r}_k/\Gm^r]$
and $X\to [\A{r}_k/\Gm^r]$ correspond to the $E^i$'s and the $D^i$'s. That
$D$ and $E$ are divisors with normal crossings is equivalent to the smoothness
of these structure morphisms. Since \cite[Theorem~1.5]{olsson_def-theory-of-stacks}
requires the base to be a scheme, we apply a different argument. The
fundamental triangle of the
composition $Y_{n-1}\to X\times \B G\to [\A{r}_k/\Gm^r]\times \B G$ together
with \cite[Theorem~1.1]{olsson_def-theory-of-stacks} gives the
exact sequence
\begin{gather*}
\Exal_{X\times \B G}(Y_{n-1},f_0^*\im^n/\im^{n+1})
\to \Exal_{[\A{r}_k/\Gm^r]\times \B G}(Y_{n-1},f_0^*\im^n/\im^{n+1})
\to\\
\to \Ext^1_{\B G}(f_0^*L_{X/k},f_0^*\im^n/\im^{n+1})=0.
\end{gather*}
We can thus promote the deformation $Y_{n-1}\inj Y_n$ over $[\A{r}_k/\Gm^r]$
to a deformation over $X$. In the final Artin approximation step, we replace
the target of the functor $F$ with maps over $[\A{r}_k/\Gm^r]$ instead of
maps over $k$.
\end{proof}

We can now prove the generalization
of~\cite[Theorem~4.1]{reichstein-youssin_ess-dim-resolution-thm}: if $(X,D)$ is a
standard pair such that the stacky locus is contained in $D$, then the kernel of
the representation $\rho_D\colon I_X\to \Gm^n\times X$ defined in \eqref{eq-divisorial-rep}
has unipotent fibers.
\begin{proof}[Proof of Theorem~\ref{TM:unipotent-stabs}]
The line bundles $\sO(D^i)$ induce one-dimensional representations of
$\stab(x)$. Let $\rho\colon \stab(x)\to \Gm^n$ be the sum of these
representations
and let $U=\ker \rho$ be its kernel. We will prove that $U$ has no non-trivial
multiplicative
subgroups. It is enough to prove that $U$ does not have a subgroup of the form
$\Gmu_r$ for any $r>1$.
It then
follows that $U$ is unipotent~\cite[Expos\'e~17, Th\'eor\`eme~4.6.1]{sga3}. Thus $\stab(x)$
is the extension of the diagonalizable group $\Delta=\image(\rho)$ by the unipotent group
$U$. Such extensions are always split since $\overline{K}$ is algebraically
closed~\cite[Expos\'e~17, Th\'eor\`eme~5.1.1]{sga3}.

Let $\Gmu_r\subseteq U$ be a finite multiplicative subgroup. By
Theorem~\ref{T:Luna-slice}, we have a flat representable morphism $p\colon
W\to X$, an \'etale stabilizer preserving morphism $q\colon W\to [\A{m}_k /
  \Gmu_r]$, and a divisor $E\subset \A{m}_k$ with simple normal crossings such
that $p^{-1}(D)=q^{-1}(E)$. Moreover, $E$ is a sum of coordinate hyperplanes
$E^i=\{x_i=0\}$ and the action of $\Gmu_r$ on $x_i$ coincides 
with the action of $\Gmu_r$ on $x^*\sO(-D_i)$ which is trivial by construction.
 Thus, the
$\Gmu_r$-fixed locus of $\A{m}_k$ is not contained in $E$. Since
$X\smallsetminus D$ is an algebraic space and $p$ is representable and
$q$ is stabilizer preserving, the action of $\Gmu_r$ on
$\A{m}_k\smallsetminus E$ is free. It follows that $r=1$.
\end{proof}

In Theorem~\ref{TM:unipotent-stabs}, we do not require that $X$ has
finite or even affine stabilizers.
Similarly, we do not impose any separatedness assumptions on $X$,
such as finite inertia, except that we require that $X$ has separated diagonal.

We note some simple consequences of Theorem~\ref{TM:unipotent-stabs}.
\begin{enumerate}
\item If $X'\to X$ denotes the $\Gm^n$-torsor corresponding to
  $\sO(D_1),\dots,\sO(D_n)$, then $X'$ has only unipotent stabilizer groups.
\item If $X$ has connected stabilizers, then the stabilizers are solvable.
\item Suppose $X$ has finite tame stabilizers, e.g., $k$ has characteristic
  zero and $X$ is Deligne--Mumford.
  Then $U=0$ so $\stab(x)=\Delta$ is diagonalizable, $X'$ is an algebraic space
  and $(X,D)$ is divisorial.
\item Suppose $k$ has characteristic $p>0$ and $X$ has only finite stabilizers.
  Then $U$ is a unipotent $p$-group: a successive extension of $\Galpha_p$'s and
  $\ZZ/p\ZZ$'s. If in addition $X$ is Deligne--Mumford, then $U$ is an \etale{}
  $p$-group. Note that $\stab(x)$ is not necessarily abelian.
\end{enumerate}

The following theorem is a generalization
of~\cite[Theorem~3.2]{reichstein-youssin_ess-dim-resolution-thm}. It gives a
different proof of Theorem~\ref{TM:divisorialification} (Divisorialification)
when $S$ is the spectrum of a field of characteristic zero,
using resolution of singularities.

\begin{theorem}\label{T:partial-torification-char0}
Let $k$ be a field of characteristic zero. Let $X$ be a smooth stack of finite
type over $k$ with separated diagonal.
Let $V\subseteq X$ be an open dense substack that is an algebraic space.
Then there exists a $V$-admissible smooth ordinary blow-up sequence
\[
\Pi\colon (Y,E)=(X_n,D_n)\to \cdots \to (X_1,D_1)\to (X_0,D_0) = (X,\emptyset)
\]
such that $Y$ has geometric stabilizers of the form
$U\rtimes \Delta$ with $U$ unipotent and $\Delta$ diagonalizable. More
precisely, the non-divisorial stabilizer $\ker(\rho_E)$ is unipotent. In
particular, if $X$ is Deligne--Mumford, then $U=0$,
$\Delta$ is finite diagonalizable and $(Y,E)$ is divisorial.
The sequence is functorial with respect to smooth stabilizer preserving
morphisms and arbitrary field extensions.
\end{theorem}
\begin{proof}
Let $W\subseteq |X|$ be the locus of points with non-trivial stabilizer. Let
$Z=\overline{W}$. Then $V\cap Z=\emptyset$ and the formation of $Z$ commutes
with smooth stabilizer preserving morphisms.

We use functorial embedded resolution of singularities to obtain a resolution
$\pi\colon Y\to X$ of $Z$. That is, $\pi$ is a smooth blow-up sequence and
$\pi^{-1}(Z)=E=\sum_i E^i$ is a divisor with simple
normal crossings. That the stabilizers have the required form follows
from Theorem~\ref{TM:unipotent-stabs}.
\end{proof}

\begin{remark}\label{R:NpS}
Let $k$ be a field of characteristic $p>0$ and let $X$ be a stack of finite
type over $k$ with finite stabilizers. We say that $X$ satisfies (NpS) if the
stabilizers of $X$ are of the form $U\rtimes H$ where $U$ is unipotent and
$H$ is linearly reductive. If $X$ is Deligne--Mumford, then equivalently
the stabilizers have normal $p$-Sylow subgroups. Theorem~\ref{TM:unipotent-stabs}
implies that if $X$ is smooth and the stackiness is contained in a divisor
with simple normal crossings, then $X$ satisfies (NpS).

Abbes and Saito prove that if $G$ is a finite group and $X=[V/G]$ has a dense
open subscheme, then there exists a normalized blow-up $Y\to X$ such that $Y$
satisfies (NpS)~\cite[Prop.~2.22]{abbes-saito_ramification-and-cleanliness}.
It would be interesting to achieve (NpS) by a smooth blow-up sequence if $X$ is
smooth. This could perhaps be a first step towards a new proof of the existence
of $p$-alterations.
\end{remark}

\end{section}


\appendix

\begin{section}{Sheaves on tame gerbes}\label{A:gerbes}
In this appendix we prove that every quasi-coherent sheaf and every complex of
quasi-coherent sheaves on a tame gerbe canonically splits as a direct sum of a
trivial part and a purely non-trivial part.

We say that $\pi\colon X\to S$ is a \emph{finite tame gerbe} if
$\pi$ is an fppf gerbe and $I_{X/S}\to X$ is finite with linearly reductive
fibers. Then $\pi$ is also
a tame coarse space: $\pi_*$ is exact and $\pi_*\sO_X=\sO_S$. In addition,
$\pi_*$ has cohomological dimension zero: $R^i\pi_*(\sF)=0$ for all $\sF\in
\QCoh(X)$ and all $i>0$. Indeed, this can be checked locally on $S$ and then
$\Dqc^+(X)=\D^+(\QCoh(X))$ because $X$ has affine diagonal~\cite[Theorem~3.8]{lurie_tannakian}.

For any gerbe $\pi\colon X\to S$ the functor $\pi^*\colon \QCoh(S)\to \QCoh(X)$
is exact and fully faithful, that is, the unit map $\sF\to
\pi_*\pi^*\sF$ is always an isomorphism. In addition, the counit map
$\pi^*\pi_*\sF\to \sF$ is always injective. Both these claims are local on $S$
so we can assume that the gerbe is neutral: $X=\B G$ for a group scheme $G\to
S$. We can then identify the category $\QCoh(X)$ with the category of
$G$-equivariant sheaves $\QCoh^G(S)$, the functor $\pi_*$ with taking
$G$-invariant subsheaves, and the functor $\pi^*$ with endowing a sheaf with
the trivial action.

\begin{definition}
Let $\pi\colon X \to S$ be a gerbe and
let $\sF\in \QCoh(X)$ be a quasi-coherent sheaf. We say that $\sF$ has
\emph{trivial action} if the counit map $\pi^*\pi_*\sF\to \sF$ is an
isomorphism. We say that $\sF$ has \emph{purely non-trivial action} if
$\pi_*\sF=0$.
We let $\QCoh_\triv(X)$ (resp.\ $\QCoh_\nt(X)$) denote the full subcategories
of $\QCoh(X)$ consisting of objects with trivial (resp.\ purely non-trivial)
action.
We also let $\sF_\triv:=\pi^*\pi_*\sF$ and $\sF_\nt:=\sF/\sF_\triv$ so that
we have an exact sequence
\[
0\to \sF_\triv \to \sF \to \sF_\nt\to 0.
\]
\end{definition}

Note that $\sF_\triv$ has trivial action but $\sF_\nt$ need not have purely
non-trivial action.

\begin{theorem}\label{T:tame-splitting:QCoh}
Let $\pi\colon X\to S$ be a finite tame gerbe.
\begin{enumerate}
\item\label{TI:tame-gerbe:Serre}
  The subcategories $\QCoh_\triv(X)$ and $\QCoh_\nt(X)$ are Serre
  subcategories.
\item\label{TI:tame-gerbe:triv}
  The functor $\pi^*$ is fully faithful with essential image $\QCoh_\triv(X)$
  and quasi-inverse $\pi_*$.
\item\label{TI:tame-gerbe:exact}
  The assignments $\sF\mapsto \sF_\triv$ and $\sF\mapsto \sF_\nt$
  give exact functors
  \begin{align*}
  \triv\colon &\QCoh(X)\to \QCoh_\triv(X),\text{ and}\\
  \nt\colon &\QCoh(X)\to \QCoh_\nt(X).
  \end{align*}
  These functors
  commute with arbitrary base change $S'\to S$.
\item\label{TI:tame-gerbe:equiv}
  The exact functor $(\triv,\nt)\colon \QCoh(X) \to
  \QCoh_\triv(X)\times \QCoh_\nt(X)$ is an equivalence of abelian categories
  with quasi-inverse given by $(\sF,\sG)\mapsto \sF\oplus \sG$.
\end{enumerate}
In particular, every quasi-coherent sheaf $\sF$ splits canonically as
$\sF_\triv\oplus \sF_\nt$.
\end{theorem}
\begin{proof}
For \itemref{TI:tame-gerbe:Serre}, let $0\to \sF'\to \sF\to \sF''\to 0$ be an
exact sequence in $\QCoh(X)$. Since $\pi_*$ and $\pi^*$ are exact and
$\pi^*\pi_*\Rightarrow \id$ is injective, it follows that $\sF$ has trivial
action (resp.\ purely non-trivial action) if and only if $\sF'$ and $\sF''$ has
such an action.

Item \itemref{TI:tame-gerbe:triv} follows from the fully faithfulness of
$\pi^*$ and the definition of $\QCoh_\triv(X)$ and holds for any gerbe.

For \itemref{TI:tame-gerbe:exact}, we note that $\triv=\pi^*\pi_*$ is a
composition of exact functors that land in $\QCoh_\triv(X)$. That $\sF_\nt$ has
purely non-trivial action follows from the exactness of $\pi_*$ and that the
functor $\nt$ is exact follows from the exactness of $\triv$. Since
$\pi_*$ and $\pi^*$ commute with arbitrary base change, so does $\triv$ and
$\nt$.

For \itemref{TI:tame-gerbe:equiv}, it remains to prove that if $\sF\in
\QCoh_\triv(X)$ and $\sG\in \QCoh_\nt(X)$, then $\Ext^i(\sF,\sG)=0$ and
$\Ext^i(\sG,\sF)=0$ for all $i$. This follows from the following lemma.
\end{proof}

\begin{lemma}
If $\sF\in \QCoh_\triv(X)$ and $\sG\in \QCoh_\nt(X)$, then
$\pi_*\sheafExt^i(\sF,\sG)=0$ and $\pi_*\sheafExt^i(\sG,\sF)=0$ for all $i$.
In particular, the exact sequence $0\to \sF_\triv\to \sF\to \sF_\nt\to 0$ is
canonically split for every $\sF\in\QCoh(X)$.
\end{lemma}
\begin{proof}
The question is local on $S$, so we may assume that $S$ is affine and $\pi$ has
a section $s\colon S\to X$. Note that for $i=0$, the question is simple: the
image of $\sF\to \sG$ is both trivial and purely non-trivial, hence $0$.  Now
consider the exact sequence
\[
0\to \sO_X\to s_*\sO_S\to Q\to 0.
\]
Since $s$ is faithfully flat, finite and of finite presentation, we have that
$Q$ is flat
and of finite presentation, i.e., a vector bundle, and $s^*\sO_X\to
s^*s_*\sO_X$ is split injective. Note that $Q=(s_*\sO_S)_\nt$. We will prove
that $\sO_X$ is a projective generator for $\QCoh_\triv(X)$ and that $Q$ is a
projective generator for $\QCoh_\nt(X)$.

If $\sF\in \QCoh_\triv(X)$, then we can find a surjection $\sO_S^{\oplus
  I}\surj \pi_* \sF$ and thus obtain a surjection $\sO_X^{\oplus I}\surj
\pi^*\pi_*\sF = \sF$.

Recall that since $s$ is finite, of finite presentation and flat, $s_*$ admits
a right adjoint $s^\times$. Moreover,
$s_*s^\times=\sheafHom_{\sO_X}(s_*\sO_S,-)$ and the counit map
$s_*s^\times\sG\to \sG$ is surjective since $\sO_X\to s_*\sO_S$ is fppf-locally
split injective.  If $\sG\in \QCoh_\nt(X)$, then we can find a surjection
$\sO_S^{\oplus I}\surj s^\times \sG$ and thus obtain a surjection
$(s_*\sO_S)^{\oplus I}\surj s_*s^\times \sG\surj \sG$. But $\sG$ is purely
non-trivial so every map $\sO_X\to \sG$ is zero. We thus obtain an induced
surjection $Q^{\oplus I}\surj \sG$.

To calculate $\pi_*\sheafExt^i(\sF,\sG)$, we may resolve $\sF$ by trivial
vector bundles and it is thus enough to calculate $\pi_*\sheafHom(\sO_X,\sG)$
which we have already seen is zero (case $i=0$). Similarly, to calculate
$\pi_*\sheafExt^i(\sG,\sF)$, we may resolve $\sG$ by vector bundles of the form
$Q^{\oplus I}$ and it is thus enough to calculate $\pi_*\sheafHom(Q,\sF)$ which
is zero (case $i=0$).
\end{proof}

We have the following derived analogue. Let $\Dqc(X)_\triv$
(resp.\ $\Dqc(X)_\nt$) be the full subcategory of $\Dqc(X)$ consisting of
objects $\sF$ such that $\pi^*\pi_*\sF\to \sF$ is an isomorphism
(resp.\ $\pi_*\sF=0$). Let $\sF_\triv=\pi^*\pi_*\sF$ and let $\sF_\nt$ be the
cone of $\sF_\triv\to \sF$.

\begin{theorem}\label{T:tame-splitting:Dqc}
Let $\pi\colon X\to S$ be a finite tame gerbe.
\begin{enumerate}
\item\label{TI:tame-gerbe-derived:localizing}
  The subcategories $\Dqc(X)_\triv$ and $\Dqc(X)_\nt$ are localizing
  subcategories.
\item\label{TI:tame-gerbe-derived:triv}
  The functor $\pi^*$ is fully faithful with essential image
  $\Dqc(X)_\triv$ and quasi-inverse $\pi_*$.
\item\label{TI:tame-gerbe-derived:orthogonal}
  There is an \emph{orthogonal decomposition}
  $\Dqc(X)=\langle \Dqc(X)_\triv, \Dqc(X)_\nt \rangle$.
\end{enumerate}
In particular, the triangle $\sF_\triv\to \sF\to \sF_\nt\to$ is split.
\end{theorem}
\begin{proof}
\ref{TI:tame-gerbe-derived:localizing} and \ref{TI:tame-gerbe-derived:triv} are
straightforward. For \ref{TI:tame-gerbe-derived:orthogonal}, we note that
$\Dqc(X)_\triv$ and $\Dqc(X)_\nt$ generate $\Dqc(X)$ since $\sF_\nt\in
\Dqc(X)_\nt$. The result thus follows from the lemma above.
\end{proof}

We will also have use of the following standard lifting result.

\begin{lemma}\label{L:deform-bundle-tame-gerbe}
Let $X\to S$ be a tame gerbe.
Let $s\in |S|$ be a point and let $\sE_0\in \VB(X_s)$ be a locally free sheaf
of finite rank on the fiber $X_s$. Then there is an \'etale neighborhood $U\to S$ of $s$ and a vector bundle
$\sE\in \VB(X\times_S U)$ such that $\sE|_s\cong \sE_0$. In addition, if
$\sE_0$ has trivial (resp.\ purely non-trivial) action, then so does $\sE$.
\end{lemma}
\begin{proof}
By a simple limit argument, we may 
assume that $(S,s)$ is local henselian. By another limit argument, we may
assume that $(S,s)$ is the henselization of a scheme of finite type over $\Spec
\ZZ$. In particular, we may assume that $S$ is excellent. Using Artin
approximation~\cite[Theorem~1.12]{artin_alg_approximation}, we may further assume that $(S,s)$ is complete local.

Let $S_n=\Spec A/\im^{n+1}$ be the $n$th infinitesimal neighborhood of $s$ in
$S=\Spec A$. Then the obstruction to lifting a vector bundle $\sE_{n-1}\in
\VB(X\times_S S_{n-1})$ to a vector bundle $\sE_n\in \VB(X\times_S S_n)$
lies in $\Ext^2(\sE_{n-1},\sE_{n-1}\otimes
\im^n/\im^{n+1})=\H^2(X_s,\sE_{n-1}^\vee\otimes \sE_{n-1}\otimes
\im^n/\im^{n+1})$ which vanishes since $X_s$ is cohomologically affine.
We can thus find a compatible system of vector bundles $(\sE_n)_{n\geq
  0}$ on $(X\times_S S_n)_{n\geq 0}$. Since $X\to S$ is proper, it follows by
Grothendieck's existence theorem for Artin
stacks~\cite{olsson_proper-coverings} that there exists a vector bundle $\sE\in
\VB(X)$ as required.

For the final claim, simply note that the splitting into trivial and
non-trivial parts commutes with arbitrary base change.
\end{proof}

The lemma also holds when $X\to S$ is the coarse space of a tame stack, although
without additional finiteness assumptions, the limit arguments become a little
more involved.

\end{section}


\begin{section}{The cotangent complex for a group of multiplicative type}
\label{app-mult-type}
In this appendix, we compute the cotangent complex for a group of multiplicative type
and for its classifying stack.
Our primary goal is to show the vanishing in $\K_0(\B G)$ of the the class of the cotangent complex $L_{\B G/k}$
of the classifying stack of a finite tame group scheme $G$ over an algebraically closed field~$k$.

Throughout the section, we fix a base algebraic stack~$S$.
It will be convenient to make the following definition.
\begin{definition}
\label{def-alg-group}
Let $S$ be an algebraic stack.
A \emph{group space} is a group object in the category of algebraic stacks over~$S$ whose
structure morphism is representable by algebraic spaces.
We say that a group space $\pi\colon G \to S$ is an \emph{algebraic group} if $\pi$ is flat and locally
of finite presentation.
\end{definition}
Note that a group space $G$ over~$S$ is algebraic if and only if the classifying stack $\B G$ is algebraic.

We start with a general discussion on cotangent complexes of algebraic groups and their classifying stacks.
Given a group space $G$ over~$S$, we consider the complex
$\Ld e^* L_{G/S}$,
where $e$ is the unit of $G$ and~$L_{G/S}$ denotes the cotangent complex.
The association
\begin{equation}
\label{eq-colie-funct}
G \mapsto C(G) := \Ld e^* L_{G/S},
\end{equation}
gives a contravariant functor from the category of group spaces
to the derived category~$\Dqc(S)$.

\begin{remark}
Note that the complex $C(G)$ generalizes the Lie algebra $\mathfrak{g}$,
which may be recovered as $\mathcal{H}^0(C(G)^\vee)$.
Also note that an exact sequence of algebraic groups
\begin{equation}
\label{eq-exact-groups}
1 \to G' \xrightarrow{f} G \xrightarrow{g} G'' \to 1
\end{equation}
induces a distinguished triangle
\begin{equation}
C(G'') \to C(G) \to C(G') \to
\end{equation}
generalizing the exact sequence of Lie algebras obtained from
\eqref{eq-exact-groups} in the case when the groups in question are smooth.
\end{remark}

\begin{proposition}
\label{prop-cotang-group}
Let $S$ be an algebraic stack and let $\pi\colon G \to S$ be an algebraic group over~$S$.
Then $C(G)$, as defined in \eqref{eq-colie-funct}, is perfect with tor-amplitude in $[-1, 0]$,
and we have
\begin{equation}
\label{eq-cotang-group}
L_{G/S} \cong \pi^*C(G).
\end{equation}
Denote the structure morphism of the classifying stack of $G$ by $\psi\colon \B G \to S$,
and let $\iota\colon I_{\B G/S} \to \B G$ be the relative inertia stack of $\psi$.
Then
\begin{equation}
\label{eq-cotang-classifying}
L_{\B G/S} \cong C(I_{\B G/S})[-1].
\end{equation}
Furthermore, 
in the special case when $G$ is abelian,
we have $L_{\B G/S} \cong \psi^*C(G)[-1]$.
\end{proposition}
\begin{proof}
Let $\tau\colon S\to \B G$ denote the universal $G$-torsor.
Using tor-independent base change applied to the cartesian square
\[
\xymatrix{
G \ar[r]^\pi\ar[d]_\pi & S \ar[d]^\tau\\
S \ar[r]_-\tau & \B G,\\
}
\]
we see that $L_{G/S} \cong \pi^* L_{S/\B G}$.
Since the unit map $e$ of $G$ is a section of $\pi$,
this gives $L_{S/\B G} \cong \Ld e^*\pi^* L_{S/\B G} \cong C(G)$,
from which \eqref{eq-cotang-group} follows.

By applying the fundamental triangle to the composition~$\psi\circ \tau = \id_S$,
we see that $\tau^*L_{\B G /S} \cong L_{S/\B G}[-1]$.
In particular, since $\psi$ is smooth, the complex $L_{\B G/S}$ is perfect with tor
amplitude in $[0, 1]$, from which the statement about the perfectness and the tor-amplitude
of $C(G)$ follows.

To prove \eqref{eq-cotang-classifying},
we apply the same argument as above to the cartesian square
$$
\xymatrix{
I_{\B G/S} \ar[r]^\iota\ar[d]_\iota & \B G \ar[d]^{\Delta_\psi}\\
\B G \ar[r]_-{\Delta_\psi} & \B G\times_S \B G,\\
}
$$
to get $L_{\Delta_{\psi}}\cong C(I_{\B G/S})$.
Since $\psi$ is flat, the cotangent complex of $\psi$ is the shift of the cotangent complex for the
diagonal morphism $\Delta_\psi$
by~$-1$ (see \cite[Proposition~17.8]{laumon}),
from which the result follows.

Finally, recall that the inertia stack of $\psi$ is isomorphic to the stack quotient $[G/G]$,
where $G$ acts on itself by conjugation.
In particular, if $G$ is abelian then $I_{\B G/S}$ is just the base change of $G$ along $\psi$.
Hence $C(I_{\B G/S}) \cong \psi^*C(G)$ by tor independent base change,
from which the last statement follows.
\end{proof}

Recall that a group space $G$ over~$S$ is of \emph{multiplicative type}
if it is isomorphic to the \emph{Cartier dual} $D(A) := \sheafHom(A, \Gm)$
of a locally constant sheaf of abelian groups  $A$ over~$S$.
A group of multiplicative type over~$S$ is algebraic in the sense of Definition~\ref{def-alg-group}
if and only if it is of finite type, which corresponds to $A$ being of finite type.
Denote the constant sheaf of rings associated to~$\ZZ$ by~$\ZZ_S$.

\begin{proposition}
\label{prop-cotang-mult-type}
Let $S$ be an algebraic stack.
Then the restriction of the functor $G \mapsto C(G)$,
as in \eqref{eq-colie-funct},
to the category of groups of multiplicative type over~$S$ is isomorphic to
$
\Delta \mapsto \mathcal{O}_S\otimes^\Ld_{\ZZ_S} D(\Delta).
$
\end{proposition}
\begin{proof}
Let $A$ be a locally constant sheaf of (not necessarily finitely generated) abelian groups,
whose group operation we denote additively.
We denote the elements of the group algebra $\mathbb{Z}_S[A]$ as formal linear combinations
of the symbols $x^\alpha$ for $\alpha \in A$,
subject to the relations $x^{\alpha + \beta} = x^\alpha x^\beta$ and $x^0 = 1$,
and we let
$$
\varepsilon_A\colon \mathbb{Z}_S[A] \to \mathbb{Z}_S, \qquad x^\alpha \mapsto 1, \qquad \alpha \in A
$$
denote the augmentation map.

Consider the endofunctor
$$
\Theta \colon A \mapsto \ZZ_S\otimes_{\ZZ_S[A]} \Omega_{\ZZ_S[A]/\ZZ_S}
$$
on the category of $\ZZ_S$-modules.
This functor acts on morphisms $\varphi\colon A \to B$ by taking the base change of
$$
\Omega_{\ZZ_S[A]/\ZZ_S} \to \Omega_{\ZZ_S[B]/\ZZ_S}, \qquad dx^\alpha \mapsto dx^{\varphi(\alpha)}, \qquad \alpha \in A,
$$
along the augmentation maps, 
which makes sense since the diagram
$$
\xymatrix{
\ZZ_S[A] \ar[drr]_{\varepsilon_A}\ar[rr]^{\ZZ_S[\varphi]} & & \ZZ_S[B] \ar[d]^{\varepsilon_B} \\
& &\ZZ_S
}
$$
commutes.
For each $A$, we also have a morphism of sheaves
$\theta_A\colon A \to \Theta(A)$
defined by taking $\alpha$ to $1\otimes dx^\alpha$,
which is natural in $A$.
Since $\ZZ_S\otimes_{\ZZ_S[A]} \Omega_{\ZZ_S[A]/\ZZ_S}$ is generated over
$\ZZ_S$ by the symbols $1\otimes dx^\alpha$ for $\alpha \in A$,
subject to the relations
$$
1\otimes dx^0 = 0,
\qquad
1\otimes d(x^{\alpha + \beta})
= 1\otimes(x^\beta dx^\alpha + x^\alpha dx^\beta)
= 1\otimes dx^\alpha + 1\otimes dx^\beta,
$$
we see that the map $\theta_A$ is in fact a group isomorphism.

Now note that the natural morphism $\mathcal{H}^0\colon L_{\ZZ_S[A]/\ZZ_S} \to \Omega_{\ZZ_S[A]/\ZZ_S}$
induces a natural transformation from the functor
$$
\Theta'\colon A \mapsto \ZZ_S \otimes^\Ld_{\ZZ_S[A]}L_{\ZZ_S[A]/\ZZ_S}
$$
to $\Theta$, if we regard the latter functor as taking values in the derived category.
Also note that both $\Theta'$ and $\Theta$ are exact and that the natural morphism $\Theta' \to \Theta$
is an isomorphism when evaluated on locally free sheaves $F$.
Hence it is actually an isomorphism for all $A$, since any locally constant sheaf of abelian groups
fits into an exact sequence
$$
0 \to F' \to F \to A \to 0
$$
with $F'$ and $F$ locally free.
The cotangent complex respects derived tor-independent base change,
so the functor
\begin{equation}
\label{eq-cotang}
A \mapsto \mathcal{O}_S \otimes^\Ld_{\mathcal{O}_S[A]}L_{\mathcal{O}_S[A]/\mathcal{O}_S}
\end{equation}
is isomorphic to $\mathcal{O}_S\otimes^{\Ld}_{\ZZ_S}\Theta'$.
Here the right hand side of \eqref{eq-cotang} is naturally isomorphic to
$\Ld e_{D(A)}^*L_{D(A)/S}$ since groups of multiplicative type are affine and the
augmentation map corresponds to the unit.
From this, we get the desired isomorphism by pre-composing \eqref{eq-cotang} with $D(-)$
and using that groups of multiplicative type are reflexive with respect to the Cartier dual.
\end{proof}

\begin{remark}
It is interesting to note that there is a moduli stack of groups of multiplicative type.
Let
$\operatorname{Core}(\operatorname{Ab}_*^{\operatorname{fg}}) \to \operatorname{Core}(\operatorname{Ab}^{\operatorname{fg}})$ be the forgetful functor from the groupoid of pointed f.g.\ abelian groups to the
groupoid of f.g.\ abelian groups.
Denote the associated morphism of constant stacks over $\Spec \ZZ$ by
$\mathcal{A} \to \mathcal{M}$.
Note that these stacks are algebraic but neither quasi-compact nor quasi-separated.
In fact, since they are constant, they are \'etale Deligne--Mumford stacks over $\Spec \ZZ$.
The stack $\mathcal{M}$ can be regarded as the moduli stack of locally constant f.g.\ abelian groups,
with $\mathcal{A}$ being the universal such group.
By Cartier duality, $\mathcal{M}$ may also be considered as the moduli stack of groups of multiplicative type,
with the universal such group given by the Cartier dual $D(\mathcal{A})$.
Denote the structure morphism by $\pi\colon D(\mathcal{A}) \to \mathcal{M}$.
Since $\pi$ is flat, the cotangent complex of any group of multiplicative type
$G \to S$ over some algebraic stack $S$ is given as the derived pull-back of
$L_{D(\mathcal{A})/\mathcal{M}}$ along the canonical map $G \to D(\mathcal{A})$.
\end{remark}

\begin{proposition}
\label{prop-const-elementary-tor}
Let $A$ be a finite abelian group endowed with an action by a finite group~$H$,
and let $k$ be an arbitrary field.
Then $\Tor_0^{\ZZ}(k, A)$
and~$\Tor_1^{\ZZ}(k, A)$
are isomorphic as $H$-representations provided that all elementary divisors of $A$
divisible by the characteristic of~$k$ are equal.
\end{proposition}
\begin{proof}
The group $A$ decomposes $H$-equivariantly into $p$-groups,
so by additivity of the tor-functors we may assume that $A$ is a $p$-group for some prime~$p$.
If the characteristic of the field is not $p$, then both tor-groups vanish and the statement is trivial.
We may therefore assume that the characteristic of $k$ is $p > 0$ and that the elementary divisors are all equal,
i.e., that $A=(\ZZ/q\ZZ)^{\oplus n}$ non-equivariantly for some $p$-power $q$.
An automorphism $\varphi$ of $A$ decomposes as $(\varphi_{ij})$
with $\varphi_{ij} \in \Hom(\ZZ/q\ZZ, \ZZ/q\ZZ)$ for $i,j=1,2,\dots,n$.
For each pair $i, j$, we have a commutative diagram
$$
\xymatrix{
0 \ar[r] & 
  \ZZ \ar[r]^q\ar[d]_{a_{ij}} &
  \ZZ \ar[r]\ar[d]_{a_{ij}} &
  \ZZ/q\ZZ \ar[r]\ar[d]_{\varphi_{ij}} &
  0 \\
0 \ar[r] & 
  \ZZ \ar[r]^q &
  \ZZ \ar[r] &
  \ZZ/q\ZZ \ar[r] &
  0 \\
}
$$ 
for some $a_{ij}$, uniquely determined modulo~$q$.
By reducing modulo~$p$ and using additivity of $\Tor_i^\ZZ(k, -)$,
we see that $\Tor_d^\ZZ(k, \varphi_{ij})\colon k \to k$ is given by
multiplication with $a_{ij}$, or equivalently, by multiplication with
$\varphi_{ij}$  for $d = 0, 1$.
This holds for any automorphism~$\varphi$,
so the group actions on $\Tor_0^{\ZZ}(k, A)$
and~$\Tor_1^{\ZZ}(k, A)$ are indeed isomorphic.
\end{proof}

\begin{example}
\label{ex-wild-tor}
The following example illustrates that the assumptions on the elementary divisors in the statement of
Proposition~\ref{prop-const-elementary-tor} cannot be removed in general.
Let $k$ be a field of characteristic~3,
and let $H$ be the cyclic group of order~$3$ acting on the group
$A = \ZZ/3\ZZ \times \ZZ/9\ZZ$ by letting
$1 \in H$ act as
$$
\begin{pmatrix}
1 & 0 \\
3 & 1 \\
\end{pmatrix}
$$
Then the action on $k^2 \cong \Tor_0^\ZZ(k, A)$ is trivial whereas $1 \in H$ acts on 
$k^2 \cong \Tor_1^\ZZ(k, A)$ as
$$
\begin{pmatrix}
1 & 0 \\
1 & 1 \\
\end{pmatrix}
$$
In particular, $\Tor_0^\ZZ(k, A)$ and $\Tor_1^\ZZ(k, A)$ are not $H$-equivariantly isomorphic.
\end{example}

\begin{remark}
\label{rem-tame-tor}
The situation in Example~\ref{ex-wild-tor} cannot occur if the order of $H$ is invertible in~$k$.
That is, in this situation $\Tor_0^{\ZZ}(k, A) \cong \Tor_1^{\ZZ}(k, A)$ $H$-equivariantly
regardless of the elementary divisors of~$A$.
This is a consequence of Maschke's theorem, which implies that $A$ splits $H$-equivariantly into 
components with constant elementary divisors, to which Proposition~\ref{prop-const-elementary-tor} applies.
\end{remark}

\begin{proposition}
\label{prop-well-split-gerbe}
Let $G = \Delta \rtimes H$ be a finite group scheme over a field~$k$,
where $\Delta$ is a diagonalizable group and $H$ is a constant group.
Then the class of the cotangent complex $L_{\B G}$ is trivial in $\K_0(\B G)$.
In particular, this applies to any finite tame group scheme over an algebraically
closed field.
\end{proposition}
\begin{proof}
Assume that $\Delta = D(A)$, where $A$ is a constant finite abelian group.
The structure morphism of the gerbe $\B G$ factors as $\B G \xrightarrow{\psi} \B H \to \Spec k$.
Since $H$ is \'etale, the cotangent complex $L_{\B H}$ vanishes,
so $L_{\B G} \cong L_{\B G/\B H}$.
By assumption, the sequence $1 \to \Delta \to G \to H \to 1$ is split,
so the morphism $\B G \to \B H$ is the classifying stack of $[\Delta/ H]$ viewed as
a group scheme over $\B H$.
In particular, we have $L_{\B G} \cong \psi^*(\sO_{\B H}\otimes^\Ld_{\ZZ_{\B H}}[A/H])[-1]$
by Proposition~\ref{prop-cotang-group} combined with Proposition~\ref{prop-cotang-mult-type}.
Thus it suffices to show that the image of the class of $[A/H]$ under the map
$\K_0(\ZZ_{\B H}) \to \K_0(\B H)$,
induced by the derived base change along $\ZZ_{\B H} \to \sO_{\B H}$, vanishes.

It is clear that $A$ decomposes $H$-equivariantly into $p$-groups for various primes~$p$.
Hence we may just as well assume that $A$ has prime power exponent $q$.
Let $A' \subset A$ be the subgroup of elements of order less than~$q$.
Then all elementary divisors of $A'' := A/A'$ are equal to~$p$.
In particular, it follows from Proposition~\ref{prop-const-elementary-tor}
that the class of $[A''/H]$ maps to zero in $\K_0(\B H)$.
By using additivity of $\K$-groups over short exact sequences together with induction on~$q$,
we see that also the class of $[A/H]$ maps to zero in $\K_0(\B H)$,
which concludes the proof.
\end{proof}


\end{section}

\bibliography{destackification}
\bibliographystyle{dary}

\end{document}